\documentclass[12pt]{amsart}

\usepackage{amstext}
\usepackage{amsthm}
\usepackage{amsmath}
\usepackage{amssymb}
\usepackage{latexsym}
\usepackage{amsfonts}
\usepackage{graphicx}
\usepackage{texdraw}
\usepackage{graphpap}
\usepackage{enumitem}
\usepackage[pagebackref,hypertexnames=false, colorlinks, citecolor=red, linkcolor=red]{hyperref} 
\usepackage[backrefs]{amsrefs}

\input txdtools

\DeclareFontFamily{U}{wncy}{}
    \DeclareFontShape{U}{wncy}{m}{n}{<->wncyr10}{}
    \DeclareSymbolFont{mcy}{U}{wncy}{m}{n}
    \DeclareMathSymbol{\Sh}{\mathord}{mcy}{"58}
\begin{document}

\newcommand{\ci}[1]{_{ {}_{\scriptstyle #1}}}

\newcommand{\norm}[1]{\ensuremath{\|#1\|}}
\newcommand{\abs}[1]{\ensuremath{\vert#1\vert}}
\newcommand{\p}{\ensuremath{\partial}}
\newcommand{\pr}{\mathcal{P}}

\newcommand{\pbar}{\ensuremath{\bar{\partial}}}
\newcommand{\db}{\overline\partial}
\newcommand{\D}{\mathcal{D}}
\newcommand{\B}{\mathbb{B}}
\newcommand{\Sp}{\mathbb{S}}
\newcommand{\T}{\mathbb{T}}
\newcommand{\R}{\mathbb{R}}
\newcommand{\Z}{\mathbb{Z}}
\newcommand{\C}{\mathbb{C}}
\newcommand{\N}{\mathbb{N}}
\newcommand{\scrH}{\mathcal{H}}
\newcommand{\scrL}{\mathcal{L}}
\newcommand{\td}{\widetilde\Delta}

\newcommand{\La}{\langle }
\newcommand{\Ra}{\rangle }
\newcommand{\rk}{\operatorname{rk}}
\newcommand{\card}{\operatorname{card}}
\newcommand{\ran}{\operatorname{Ran}}
\newcommand{\osc}{\operatorname{OSC}}
\newcommand{\im}{\operatorname{Im}}
\newcommand{\re}{\operatorname{Re}}
\newcommand{\tr}{\operatorname{tr}}
\newcommand{\vf}{\varphi}
\newcommand{\f}[2]{\ensuremath{\frac{#1}{#2}}}


\newcommand{\entrylabel}[1]{\mbox{#1}\hfill}

\newenvironment{entry}
{\begin{list}{X}%
  {\renewcommand{\makelabel}{\entrylabel}%
      \setlength{\labelwidth}{55pt}%
      \setlength{\leftmargin}{\labelwidth}
      \addtolength{\leftmargin}{\labelsep}%
   }%
}%
{\end{list}}


\numberwithin{equation}{section}
\newtheorem{dfn}{Definition}[section]
\newtheorem{thm}{Theorem}[section]
\newtheorem{lm}[thm]{Lemma}
\newtheorem{cor}[thm]{Corollary}
\newtheorem{conj}[thm]{Conjecture}
\newtheorem{prob}[thm]{Problem}
\newtheorem{prop}[thm]{Proposition}
\newtheorem*{prop*}{Proposition}

\theoremstyle{remark}
\newtheorem{rem}[thm]{Remark}
\newtheorem*{rem*}{Remark}

\newtheorem{quest}[thm]{Question}

\title{Multilinear Dyadic Operators And Their Commutators}

\author{Ishwari Kunwar}

\address{Ishwari Kunwar, School of Mathematics\\ Georgia Institute of Technology\\ 686 Cherry Street\\ Atlanta, GA USA 30332-0160}
\email{ikunwar3@math.gatech.edu}

\subjclass[2000]{Primary }

\begin{abstract} 
We introduce multilinear analogues of dyadic paraproduct operators and Haar Multipliers, and study boundedness properties of these operators and their commutators. We also characterize dyadic $BMO$ functions via boundedness of certain paraproducts and also via boundedness of the commutators of multilinear Haar Multipliers and paraproduct operators. 
\end{abstract}

\maketitle
\setcounter{tocdepth}{1}
\tableofcontents
\section{Introduction and statement of main results}
\noindent
Dyadic operators have attracted a lot of attention in the recent years. The proof of so-called $A_2$ theorem (see \cite{Hyt}) consisted in representing a general Calder$\acute{\text{o}}$n-Zygmund operator as an average of dyadic shifts, and then verifying some testing conditions for those simpler dyadic operators. It seems reasonable to  believe that, taking a similar approach, general multilinear Calder$\acute{\text{o}}$n-Zygmund operators can be studied by studying multilinear dyadic operators. Regardless of this possibility, multilinear dyadic operators in their own right are an important class of objects in Harmonic Analysis. Statements regarding those operators can be translated into the non-dyadic world, and are sometimes simpler to prove.\\

\noindent  
In this paper we introduce multilinear analogues of dyadic operators such as paraproducts and Haar multipliers, and study their boundedness properties. Corresponding theory of linear dyadic operators, which we will be using very often, can be found in \cite{Per}. In \cite{BMNT}, the authors have studied boundedness properties of bilinear paraproducts defined in terms of so-called ``smooth molecules". The paraproduct operators we study are more general multilinear operators, but defined in terms of indicators and Haar functions of dyadic intervals. In \cite{CRW} Coifman, Rochberg and Weiss proved that the commutator of a $BMO$ function with a singular integral operator is bounded in $L^p$, $1<p<\infty.$ The necessity of $BMO$ condition for the boundedness of the commutator was also established for certain singular integral operators, such as the Hilbert transform. S. Janson \cite{Jan} later studied its analogue for linear martingale transforms. In this paper we study commutators of multilinear dyadic operators, and characterize dyadic $BMO$ functions via boundedness of these commutators. For the corresponding theory for general multilinear Calder$\acute{\text{o}}$n-Zygmund operators we refer to \cite{GT} and \cite{LOPTT}.\\
 
\noindent
We organize the paper as follows:\\

\noindent
In section 2, we present an overview of some of the main tools we will be using in this paper. These include: the Haar system, linear Haar multipliers, dyadic maximal/square functions, linear/bilinear paraproduct operators and the space of dyadic $BMO$ functions. For more details we refer to \cite{Per}.\\

\noindent
In section 3, we obtain a decomposition of the pointwise product of $m$ functions, $m \geq 2,$ which generalizes the paraproduct decomposition of two functions.  On the basis of this decomposition we define multilinear paraproducts and investigate their boundedness properties as operators on products of Lebesgue spaces. We also define multilinear anologue of the linear paraproduct operator $\pi_b$, and characterize dyadic $BMO$ functions via boundedness of certain multilinear paraproduct operators.\\

\noindent
In section 4, we define multilinear Haar multipliers in a way consistent with the definition of linear Haar multipliers and multilinear paraproducts, and then investigate their boundedness properties. We also study boundedness properties of their commutators with dyadic $BMO$ functions, and provide a characterization of dyadic $BMO$ functions via the boundedness of those multilinear commutators. In particular, we show that the commutators of the multilinear paraproducts with a function $b$ are bounded if and only if $b$ is a dyadic $BMO$ function. \\

\noindent
Our main results involve the following operators:\\
\begin{itemize}
	
\item  $\displaystyle P^{\vec{\alpha}}(f_1,f_2,\ldots,f_m) = \sum_{I\in\mathcal{D}} \left(\prod_{j=1}^m f_j(I,\alpha_j)\right) h_I^{\sigma(\vec{\alpha})}, \quad \vec{\alpha} \in \{0,1\}^m \backslash\{(1,1,\ldots,1)\}. $\\

\item $\displaystyle \pi_b^{\vec{\alpha}}(f_1, f_2, \ldots, f_m) = \sum_{I \in \mathcal{D}} \La b , h_I \Ra \left(\prod_{j=1}^m f_j(I,\alpha_j)\right)  h_I^{1+\sigma(\vec{\alpha})},\quad \vec{\alpha} \in \{0,1\}^m.$\\

\item  $\displaystyle T_\epsilon^{\vec{\alpha}} (f_1,f_2, \ldots,f_m) := \sum_{I\in \D} \epsilon_I \left(\prod_{j=1}^m f_j(I,\alpha_j)\right) h_I^{\sigma(\vec{\alpha})},$ \\

 $ \quad \vec{\alpha} \in =\{0,1\}^m \backslash \{(1,1,\ldots,1)\}, \, \epsilon = \{\epsilon_I\}_{I\in \D} \text{ bounded}.$\\

\item $\displaystyle [b,T_\epsilon^{\vec{\alpha}}]_i(f_1,f_2,\ldots,f_m)(x) := b(x)T_\epsilon^{\vec{\alpha}}(f_1,f_2,\ldots,f_m)(x) - T_\epsilon^{\vec{\alpha}}(f_1, \ldots, bf_i,\ldots,f_m)(x),$\\

 $ 1 \leq i \leq m$, $ \vec{\alpha} \in \{0,1\}^m \backslash\{(1,1,\ldots,1)\},\, \epsilon = \{\epsilon_I\}_{I\in \D} \text{ bounded and } b\in BMO^d.$\\
\end{itemize}

\noindent In the above definitions, $\D := \{[m2^{-k}, (m+1)2^{-k}): m,k\in \mathbb{Z}\}$ is the standard dyadic grid on $\R$ and $h_I$'s are the Haar functions defined by $h_I = \displaystyle \frac{1}{\abs{I}^{1/2}}\left(\mathsf{1}_{I_+} - \mathsf{1}_{I_-}\right),$ where $I_-$ and $I_+$ are the left and right halves of $I.$ With $\left< \;,\;\right>$ denoting the standard inner product in $L^2(\R),$ $f_i(I,0) := \left< f_i,h_I\right>$ and $\displaystyle f_i(I,1) := \La f_i, h_I^2\Ra = \frac{1}{\abs{I}} \int_I f_i,$ the average of $f_i$ over $I.$ The Haar coefficient $\La f_i, h_I\Ra$ is sometimes denoted by $\widehat{f_i}(I)$ and the average of $f_i$ over $I$ by $\La f_i \Ra_I$. For $\vec{\alpha} \in \{0,1\}^m,$ $\sigma(\vec{\alpha})$ to denotes the number of 0 components in $\vec{\alpha}$. For convenience, we will denote the set $\{0,1\}^m \backslash\{(1,1,\ldots,1)\}$ by $U_m.$\\

\noindent
In the following main results $L^p$ stands for the Lebesgue space $L^p(\R):= \left\{f:\norm{f}_p < \infty \right\} $ with $\displaystyle\norm{f}_p = \norm{f}_{L^p} := \left(\int_\R \abs{f(x)}^p dx \right)^{1/p}.$ The Weak $L^p$ space, also denoted by $L^{p,\infty}$, is the space of all functions $f$ such that
$$ \norm{f}_{L^{p,\infty}(\mathbb{R})}:= \sup_{t>0} t \, \left\vert \{x\in \mathbb{R}: f(x) >t \} \right\vert^{1/p} < \infty.$$
Moreover, $ \displaystyle \norm{b}_{BMO^d}:=\sup_{I\in \D}\frac{1}{\abs{I}}\int_I \abs{b(x) - \La b \Ra_I} \,dx < \infty, $ is the dyadic $BMO$ norm of $b.$\\

\noindent 
We now state our main results:\\

\noindent
\noindent
\textbf{Theorem:} Let $ \vec{\alpha} = (\alpha_1,\alpha_2,\ldots,\alpha_m) \in \{0,1\}^m$ and $ 1 < p_1, p_2, \ldots, p_m < \infty$ with $\displaystyle \sum_{j=1}^m \frac{1}{p_j} = \frac{1}{r}.$ Then
\begin{enumerate}[label = $(\alph*)$]
\item For $\vec{\alpha}  \neq (1,1,\ldots,1),$  $\displaystyle \left\Vert P^{\vec{\alpha}}(f_1,f_2,\ldots,f_m)\right\Vert_r \lesssim \prod_{j=1}^m\norm{f_j}_{p_j}.$
\item For $\sigma(\vec{\alpha}) \leq 1,$ $\displaystyle \left\Vert\pi_b^{\vec{\alpha}}(f_1,f_2,\ldots,f_m)\right\Vert_r \lesssim \norm{b}_{BMO^d}\prod_{j=1}^m\norm{f_j}_{p_j},$ if and only if $b \in BMO^d.$\\

\item For $\sigma(\vec{\alpha}) > 1,$ $\displaystyle\left\Vert\pi_b^{\vec{\alpha}}(f_1,f_2,\ldots,f_m)\right\Vert_r \leq C_b \prod_{j=1}^m\norm{f_j}_{p_j},$ if and only if $\displaystyle\sup_{I\in \D} \frac{\abs{\La b,h_I\Ra}}{\sqrt{\abs{I}}} < \infty.$
\end{enumerate}
In each of the above cases, the paraproducts are weakly bounded if $1\leq p_1, p_2, \ldots, p_m < \infty$.\\

\noindent
\textbf{Theorem:}  Let $\epsilon = \{\epsilon_I\}_{I\in\D}$ be a given sequence and let $\vec{\alpha} = (\alpha_1,\alpha_2, \ldots,\alpha_m) \in U_m.$  Let $1<p_1,p_2, \ldots,p_m<\infty$ with
$$\displaystyle \sum_{j=1}^m \frac{1}{p_j} = \frac{1}{r}.$$
Then $T_\epsilon^{\vec{\alpha}}$ is bounded from $L^{p_1}\times L^{p_2} \times \cdots\times L^{p_m}$ to $L^r$ if and only if $\norm{\epsilon}_\infty:= \displaystyle \sup_{I \in \D}\abs{\epsilon_I} < \infty.$\\
Moreover, $T_\epsilon^{\vec{\alpha}}$ has the corresponding weak-type boundedness if $1 \leq p_1,p_2, \ldots,p_m<\infty.$\\

\noindent
\textbf{Theorem:} Let $\vec{\alpha} = (\alpha_1,\alpha_2,\ldots,\alpha_m) \in U_m,$ $1\leq i \leq m,$ and $1<p_1,p_2, \ldots,p_m, r < \infty$ with 
$$\sum_{j=1}^m \frac{1}{p_j} = \frac{1}{r}.$$ Suppose $b \in L^p$ for some $p \in (1,\infty).$ Then the following two statements are equivalent.
\begin{enumerate}[label = $(\alph*)$]
\item  $b\in BMO^d.$\\
\item  $\displaystyle [b,T_\epsilon^{\vec{\alpha}}]_i:L^{p_1}\times L^{p_2} \times \cdots\times L^{p_m}\rightarrow L^r $ is bounded for every bounded sequence $\epsilon = \{\epsilon_I\}_{I\in \D}.$\\
\end{enumerate}

\noindent
In particular,  $b\in BMO^d$ if and only if
$[b,P^{\vec{\alpha}}]_i:L^{p_1}\times L^{p_2} \times \cdots\times L^{p_m}\rightarrow L^r$ is bounded.\\

\noindent
\textbf{Acknowledgement:} The author would like to thank Brett Wick for suggesting him this research project, and for providing valuable suggestions.

\section{Notation and preliminaries}

\subsection{The Haar System and the Haar multipliers:}
Let $\D$ denote the standard dyadic grid on $\R,$  $$\D = \{[m2^{-k}, (m+1)2^{-k}): m,k\in \mathbb{Z}\}.$$
Associated to each dyadic interval $I$ there is a Haar function $h_I$ defined by
$$h_I(x) = \frac{1}{\abs{I}^{1/2}}\left(\mathsf{1}_{I_+} - \mathsf{1}_{I_-}\right),$$
where $I_-$ and $I_+$ are the left and right halves of $I.$\\

\noindent
The collection of all Haar functions $\{h_I: I \in \D\}$ is an orthonormal basis of $L^2(\R),$ and an unconditional basis of $L^p$ for $ 1 < p < \infty.$ In fact, if a sequence $\epsilon = \{\epsilon_I\}_{I \in \mathcal{D}}$ is bounded, the operator $T_\epsilon$ defined by
$$T_\epsilon f(x) = \sum_{I \in \mathcal{D}} \epsilon_I \La f, h_I \Ra  h_I $$
is bounded in $L^p$ for all $1 < p  < \infty.$ The converse also holds. The operator $T_\epsilon$ is called the Haar multiplier with symbol $\epsilon.$ \\

\subsection{The dyadic maximal function:}
Given a function $f$, the dyadic Hardy-Littlewood maximal function $M^df$ is defined by
$$M^df(x):= \sup_{x\in I\in \D} \frac{1}{\abs{I}} \int_I \abs{f(t)}\,dt.$$

\noindent
For the convenience of notation, we will just write $M$ to denote the dyadic maximal operator. 
Clearly, $M$ is bounded on $L^\infty.$ It is well-known that $M$ is of weak type $(1,1)$ and strong type $(p,p)$ for all $1<p<\infty.$\\

\subsection{The dyadic square function:}
The dyadic Littlewood-Paley square function of a function $f$ is defined by
$$S f(x):= \left(\sum_{I \in \D} \frac{\abs{\La f,h_I \Ra}^2}{\abs{I}} \mathsf{1}_I(x) \right)^{1/2}.$$
For $f\in L^p$ with $1<p<\infty,$ we have $\norm{Sf}_p \approx \norm{f}_p$ with equality when $p=2.$\\

\subsection{BMO Space} A locally integrable function $b$ is said to be of bounded mean oscillation if
$$\norm{b}_{BMO}:=\sup_{I}\frac{1}{\abs{I}}\int_I \abs{b(x) - \La b \Ra_I} \,dx < \infty, $$
where the supremum is taken over all intervals in $\mathbb{R}.$ The space of all functions of bounded mean oscillation is denoted by $BMO.$\\

\noindent
If we take the supremum over all dyadic intervals in $\mathbb{R},$ we get a larger space of dyadic BMO functions which we denote by $BMO^d.$\\

\noindent
For $0<r<\infty,$ define
$$ BMO_r = \left\{b \in L_{loc}^r(\mathbb{R}): \norm{b}_{BMO_r} < \infty \right\},$$
where, $\displaystyle \norm{b}_{BMO_r} := \left(\sup_{I}\frac{1}{\abs{I}}\int_I \abs{b(x) - \La b \Ra_I}^r \,dx \right)^{1/r}.$\\

\noindent For any $0<r<\infty,$ the norms $\norm{b}_{BMO_r}$ and $\norm{b}_{BMO}$ are equivalent. The equivalence of norms for $r > 1$ is well-known and follows from John-Nirenberg's lemma (see \cite{JN}), while the equivalence for $0<r<1$ has been proved by Hanks in \cite{HR}. (See also \cite{SE}, page 179.)\\

\noindent
For $r=2$, it follows from the orthogonality of Haar system that 
$$ \norm{b}_{BMO_2^d} = \left(\sup_{I \in \D} \frac{1}{\abs{I}} \sum_{J \subseteq I} \abs{\widehat{b}(J)}^2\right)^{1/2}.$$

\subsection{The linear/ bilinear paraproducts:}
Given two functions $f_1$ and $f_2$, the point-wise product $f_1f_2$ can be decomposed into the sum of bilinear paraproducts:
 $$ f_1f_2 = P^{(0,0)}(f_1,f_2) + P^{(0,1)}(f_1,f_2) + P^{(1,0)}(f_1,f_2),$$
 where for $\vec{\alpha} = (\alpha_1, \alpha_2) \in \{0,1\}^2$,
 $$ P^{\vec{\alpha}}(f_1,f_2) = \displaystyle \sum_{I \in \mathcal{D}} f_1(I,\alpha_1) f_2(I, \alpha_2) h_I^{\sigma(\vec{\alpha})}$$
with $ f_i(I,0) = \La f_i, h_I \Ra, \;\; f_i(I,1) = \La f_i \Ra_I,\;  \sigma(\vec{\alpha}) = \# \{i: \alpha_i = 0\},  \text{ and }  h_I^{\sigma(\vec{\alpha})}$ being the pointwise product $h_I h_I \ldots h_I$ of $ \sigma(\vec{\alpha})$ factors. \\

The paraproduct $P^{(0,1)}(f_1,f_2)$ is also denoted by $\pi_{f_1}(f_2),$ i.e.,\\
$$ \pi_{f_1}(f_2) = \sum_{I \in \mathcal{D}} \La f_1, h_I \Ra \La f_2 \Ra_I h_I.$$
Observe that $$\La \pi_{f_1}(f_2), g \Ra = \left\La \sum_{I \in \mathcal{D}} \La f_1, h_I \Ra \La f_2 \Ra_I h_I, g \right\Ra = \sum_{I \in \mathcal{D}} \La f_1, h_I \Ra \La f_2 \Ra_I \La g, h_I \Ra $$
which is equal to \begin{eqnarray*} \left\La f_2, P^{(0,0)}(f_1,g) \right\Ra
& = & \left\La f_2, \;\sum_{I \in \mathcal{D}} \La f_1, h_I \Ra \La g, h_I \Ra h_I^2 \right\Ra \\
& = & \sum_{I \in \mathcal{D}} \La f_1, h_I \Ra \La g,h_I \Ra \La f_2, h_I^2 \Ra \\
& = & \sum_{I \in \mathcal{D}} \La f_1, h_I \Ra \La f_2 \Ra_I \La g, h_I \Ra .
\end{eqnarray*}

\noindent
This shows that $\pi_{f_1}^* = P^{(0,0)}(f_1, \cdot) = P^{(0,0)}(\cdot,f_1)$.\\

\noindent
The ordinary multiplication operator $M_b: f \rightarrow bf$ can therefore be given by:
$$M_b(f) = bf = P^{(0,0)}(b,f) + P^{(0,1)}(b,f) + P^{(1,0)}(b,f) = \pi_b^*(f) + \pi_b(f) + \pi_f(b).$$

\noindent
The function $b$ is required to be in $L^\infty$ for the boundedness of $M_b$ in $L^p$.  However, the paraproduct operator $\pi_b$ is bounded in $L^p$ for every $1 < p < \infty$ if $b \in BMO^d.$ Note that $BMO^d$ properly contains $L^\infty$. Detailed information on the operator $\pi_b$ can be found in \cite{Per} or \cite{Bla}.

\subsection{Commutators of Haar multipliers:}
The commutator of $T_\epsilon$ with a locally integrable function $b$ is defined by
$$ [b, T_\epsilon](f)(x) := T_\epsilon(bf)(x) - M_b (T_\epsilon(f))(x).$$
\noindent
It is well-known that for a bounded sequence $\epsilon$ and $1<p<\infty$, the commutator $[b, T_\epsilon]$ is bounded in $L^p$ for all $p\in (1, \infty)$  if $b \in BMO^d.$\\
These commutators have been studied in \cite{Treil} in non-homogeneous martingale settings.\\

\section{Multilinear dyadic paraproducts}

\subsection{Decomposition of pointwise product $\displaystyle\prod_{j=1}^m f_j$}

 \noindent
In this sub-section we obtain a decomposition of pointwise product $\displaystyle\prod_{j=1}^m f_j$ of $m$ functions that is analogous to the following paraproduct decomposition :
$$ f_1f_2 = P^{(0,0)}(f_1,f_2) + P^{(0,1)}(f_1,f_2) + P^{(1,0)}(f_1,f_2).$$
The decomposition of $\displaystyle\prod_{j=1}^m f_j$ will be the basis for defining \textit{multi-linear paraproducts} and \textit{m-linear Haar multipliers}, and will also be very useful in proving boundedness properties of multilinear commutators.\\

\noindent
We first introduce the following notation:\\
\begin{itemize}
	\item $f(I,0) := \widehat{f}(I) = \La f,h_I \Ra = \displaystyle \int_\R f(x) h_I(x) dx. $
	\item $f(I,1) := \La f \Ra_I = \frac{1}{\abs{I}} \displaystyle \int_I f(x) dx. $\\	
	\item $U_m:= \left\{(\alpha_1, \alpha_2, \ldots,\alpha_m) \in \{0,1\}^m: (\alpha_1, \alpha_2,\ldots,\alpha_m) \neq (1,1, \ldots,1)\right\}.$\\
	
	\item $\sigma(\vec{\alpha}) = \# \{i: \alpha_i =0\}$ for $\vec{\alpha} = (\alpha_1, \ldots,\alpha_m) \in \{0,1\}^m.$\\
	\item $ (\vec{\alpha}, i) = (\alpha_1, \ldots,\alpha_m, i),\; (i,\vec{\alpha}) = (i,\alpha_1, \ldots,\alpha_m)$ for $\vec{\alpha} = (\alpha_1, \ldots,\alpha_m) \in \{0,1\}^m.$\\
	
	\item $P_I^{\vec{\alpha}} (f_1,  \ldots,f_m) = \prod_{j=1}^m f_j(I,\alpha_j)h_I^{\sigma(\vec{\alpha})}$ for $\vec{\alpha} \in U_m$ and $I \in \D.$\\
	\item $P^{\vec{\alpha}}(f_1,  \ldots,f_m) = \displaystyle\sum_{I\in\D} P_I^{\vec{\alpha}} (f_1, \ldots,f_m) = \displaystyle\sum_{I\in\D}\prod_{j=1}^m f_j(I,\alpha_j)h_I^{\sigma(\vec{\alpha})}$ for $\vec{\alpha} \in U_m.$
	\end{itemize}
	
	\noindent
	With this notation, the paraproduct decomposition of $f_1f_2$ takes the following form:
	$$ f_1f_2 = P^{(0,0)}(f_1,f_2) + P^{(0,1)}(f_1,f_2) + P^{(1,0)}(f_1,f_2) = \sum_{\vec{\alpha} \in U_2} P^{\vec{\alpha}}(f_1,f_2).\\ $$ 
	Note that \begin{equation} \label{IndexSetUm} U_m = \{(\alpha,1): \vec{\alpha} \in U_{m-1}\} \cup \{(\vec{\alpha},0): \vec{\alpha} \in U_{m-1}\} \cup \{(1,\ldots,1,0)\}.
\end{equation}
	\noindent
	To obtain an analogous decomposition of $\displaystyle\prod_{j=1}^m f_j,$ we need the following crucial lemma:
	
	\begin{lm}
	Given $m\geq 2$ and functions $f_1,f_2, \ldots, f_m,$ with $f_i \in L^{p_i}, 1<p_i<\infty,$we have
	$$\prod_{j=1}^{m} \La f_j \Ra_J \mathsf{1}_J = \sum_{\vec{\alpha} \in U_m} \sum_{J\subsetneq I}P_I^{\vec{\alpha}} (f_1,f_2, \ldots, f_m) \; \mathsf{1}_J, $$
	for all $J\in\D.$
	\end{lm}
	
	\noindent
	\begin{proof} We prove the lemma by induction on $m.$\\
	
	\noindent
	First assume that $m=2.$ We want to prove the following:
	\begin{eqnarray} \label{AverageProduct}
\La f_1 \Ra_J \La f_2\Ra_J\mathsf{1}_J&=& \sum_{\vec{\alpha} \in U_2} \sum_{J\subsetneq I}P^{\vec{\alpha}}_I(f_1,f_2) \; \mathsf{1}_J\\
&=& \nonumber\left(\sum_{J\subsetneq I}P^{(0,1)}_I(f_1,f_2)+\sum_{J\subsetneq I}P^{(1,0)}_I(f_1,f_2)+\sum_{J\subsetneq I}P^{(0,0)}_I(f_1,f_2)\right) \; \mathsf{1}_J\\
&=& \nonumber \left(\sum_{J\subsetneq I}\widehat{f_1}(I)\La f_2 \Ra_I h_I+\sum_{J\subsetneq I}\La f_1 \Ra_I \widehat{f_2}(I)h_I+\sum_{J\subsetneq I}\widehat{f_1}(I)\widehat{f_2}(I) h_I^2\right) \; \mathsf{1}_J.
\end{eqnarray}
\noindent
We have,
\begin{eqnarray*}
&&\La f_1 \Ra_J \left< f_2 \right>_J \mathsf{1}_J\\
&=&  \left( \sum_{J\subsetneq I}\widehat{f_1}(I) h_I\right) \left( \sum_{J\subsetneq K}\widehat{f_2}(K) h_K\right)\mathsf{1}_J\\
&=& \sum_{J\subsetneq I}\widehat{f_1}(I) h_I \left( \sum_{I\subsetneq K}\widehat{f_2}(K) h_K + \widehat{f_2}(I) h_I+\sum_{J\subsetneq K \subsetneq I}\widehat{f_2}(K) h_K\right)\mathsf{1}_J\\
&=& \left\{\sum_{J\subsetneq I}\widehat{f_1}(I) \left< f_2\right>_I h_I + \sum_{J\subsetneq I}\widehat{f_1}(I) \widehat{f_2}(I) h_I^2 + \sum_{J\subsetneq I}\widehat{f_1}(I) h_I \left(\sum_{J\subsetneq K \subsetneq I}\widehat{f_2}(K) h_K\right)\right\}\mathsf{1}_J\\
&=& \left\{\sum_{J\subsetneq I}\widehat{f_1}(I) \left< f_2\right>_I h_I + \sum_{J\subsetneq I}\widehat{f_1}(I) \widehat{f_2}(I) h_I^2 + \sum_{J\subsetneq K}\widehat{f_2}(K) h_K \left(\sum_{K \subsetneq I}\widehat{f_1}(I) h_I\right)\right\}\mathsf{1}_J\\
&=& \left\{\sum_{J\subsetneq I}\widehat{f_1}(I) \left< f_2\right>_I h_I + \sum_{J\subsetneq I}\widehat{f_1}(I) \widehat{f_2}(I) h_I^2 + \sum_{J\subsetneq K}\widehat{f_2}(K) \left<f_1\right>_K h_K \right\}\mathsf{1}_J\\
&=& \left\{\sum_{J\subsetneq I}\widehat{f_1}(I) \left< f_2\right>_I h_I + \sum_{J\subsetneq I}\widehat{f_1}(I) \widehat{f_2}(I) h_I^2 + \sum_{J\subsetneq I}\widehat{f_2}(I)\left<f_1\right>_Ih_I \right\}\mathsf{1}_J\\
&=& \left(\sum_{J\subsetneq I}\widehat{f_1}(I)\La f_2 \Ra_I h_I+\sum_{J\subsetneq I}\La f_1 \Ra_I \widehat{f_2}(I)h_I+\sum_{J\subsetneq I}\widehat{f_1}(I)\widehat{f_2}(I) h_I^2\right) \; \mathsf{1}_J.\\
\end{eqnarray*}

\noindent
Now assume  $m > 2$ and that
 $$\prod_{j=1}^{m-1} \La f_j \Ra_J \mathsf{1}_J = \sum_{\vec{\alpha} \in U_{m-1}} \sum_{J\subsetneq I}P_I^{\vec{\alpha}} (f_1,f_2, \ldots, f_{m-1})  \mathsf{1}_J. $$

\noindent
Then, 
\begin{eqnarray*}
&&\prod_{j=1}^{m} \La f_j \Ra_J \mathsf{1}_J\\
&=& \left(\prod_{j=1}^{m-1} \La f_j \Ra_J \mathsf{1}_J \right) \La f_m\Ra_J\mathsf{1}_J\\
&=& \sum_{\vec{\alpha} \in U_{m-1}} \sum_{J\subsetneq I}P_I^{\vec{\alpha}} (f_1,f_2, \ldots, f_{m-1}) \left( \sum_{J\subsetneq K}\widehat{f_m}(K) h_K\right)  \mathsf{1}_J\\
&=& \sum_{\vec{\alpha} \in U_{m-1}} \sum_{J\subsetneq I}P_I^{\vec{\alpha}} (f_1,f_2, \ldots, f_{m-1})\left( \sum_{I\subsetneq K}\widehat{f_m}(K) h_K + \widehat{f_m}(I) h_I+\sum_{J\subsetneq K \subsetneq I}\widehat{f_m}(K) h_K\right)\mathsf{1}_J\\
\end{eqnarray*}

\noindent
This gives
\begin{eqnarray*}
&&\prod_{j=1}^{m} \La f_j \Ra_J \mathsf{1}_J\\
&=& \sum_{\vec{\alpha} \in U_{m-1}} \sum_{J\subsetneq I}P_I^{\vec{\alpha}} (f_1,f_2, \ldots, f_{m-1})\La f_m \Ra_I \mathsf{1}_J + \sum_{\vec{\alpha} \in U_{m-1}} \sum_{J\subsetneq I}P_I^{\vec{\alpha}} (f_1,f_2, \ldots, f_{m-1}) \widehat{f_m}(I) h_I \mathsf{1}_J\\
 && \quad +\sum_{\vec{\alpha} \in U_{m-1}} \sum_{J\subsetneq I}P_I^{\vec{\alpha}} (f_1,f_2, \ldots, f_{m-1})\left( \sum_{J\subsetneq K \subsetneq I}\widehat{f_m}(K) h_K\right)\mathsf{1}_J\\
&=& \sum_{\vec{\alpha} \in U_{m-1}} \sum_{J\subsetneq I}P_I^{(\vec{\alpha},1)} (f_1,f_2, \ldots, f_m) \mathsf{1}_J + \sum_{\vec{\alpha} \in U_{m-1}} \sum_{J\subsetneq I}P_I^{(\vec{\alpha},0)} (f_1,f_2, \ldots, f_m) \mathsf{1}_J\\
&& \quad+  \sum_{J\subsetneq K}\widehat{f_2}(K) h_K \left(\sum_{\vec{\alpha} \in U_{m-1}}\sum_{K\subsetneq I}P_I^{\vec{\alpha}} (f_1,f_2, \ldots, f_{m-1})\right) \mathsf{1}_J\\
&=& \sum_{\vec{\alpha} \in U_{m-1}} \sum_{J\subsetneq I}P_I^{(\vec{\alpha},1)} (f_1,f_2, \ldots, f_m) \mathsf{1}_J + \sum_{\vec{\alpha} \in U_{m-1}} \sum_{J\subsetneq I}P_I^{(\vec{\alpha},0)} (f_1,f_2, \ldots, f_m) \mathsf{1}_J\\
 && \quad+  \sum_{J\subsetneq K}\widehat{f_m}(K) h_K \La f_1\Ra_K \ldots\La f_{m-1}\Ra_K \mathsf{1}_J\\
&=&\sum_{\vec{\alpha} \in U_{m-1}} \sum_{J\subsetneq I}P_I^{(\vec{\alpha},1)} (f_1,f_2, \ldots, f_m) \mathsf{1}_J + \sum_{\vec{\alpha} \in U_{m-1}} \sum_{J\subsetneq I}P_I^{(\vec{\alpha},0)} (f_1,f_2, \ldots, f_m) \mathsf{1}_J\\
&& \quad + \sum_{J\subsetneq I} P_I^{(1,\ldots,1,0)}(f_1,f_2,\ldots,f_m) \mathsf{1}_J\\
&=& \sum_{\vec{\alpha} \in U_m} \sum_{J\subsetneq I}P_I^{\vec{\alpha}} (f_1,f_2, \ldots, f_m) \mathsf{1}_J.
\end{eqnarray*}
The last equality follows from (\ref{IndexSetUm}).
\end{proof}

	\begin{lm}
	Given $m\geq 2$ and functions $f_1,f_2, \ldots, f_m,$ with $f_i \in L^{p_i}, 1<p_i<\infty,$we have
	$$\displaystyle\prod_{j=1}^m f_j = \sum_{\vec{\alpha} \in U_m} P^{\vec{\alpha}}(f_1,f_2, \ldots, f_m). $$
	\end{lm}
	
\noindent
\begin{proof} We have already seen that it is true for $m=2.$ By induction, assume that
\begin{eqnarray*}
\displaystyle\prod_{j=1}^{m-1} f_j &=& \sum_{\vec{\alpha} \in U_{m-1}} P^{\vec{\alpha}}(f_1,f_2, \ldots, f_{m-1})\\
&= & \sum_{\vec{\alpha} \in U_{m-1}} \sum_{I\in\D} P_I^{\vec{\alpha}}(f_1,f_2, \ldots, f_{m-1})
\end{eqnarray*}

\noindent
Then,
\begin{eqnarray*}
\displaystyle\prod_{j=1}^m f_j &=& \left(\displaystyle\prod_{j=1}^{m-1} f_j \right) f_m\\
&=& \sum_{\vec{\alpha} \in U_{m-1}} \sum_{I\in\D} P_I^{\vec{\alpha}}(f_1,f_2, \ldots, f_{m-1}) \left(\sum_{J\in\D}\widehat{f_m}(J) h_J \right)\\
&=& \sum_{\vec{\alpha} \in U_{m-1}} \sum_{I\in\D} P_I^{\vec{\alpha}}(f_1,f_2, \ldots, f_{m-1}) \left( \sum_{I\subsetneq J}\widehat{f_m}(J) h_J + \widehat{f_m}(I) h_I+\sum_{J\subsetneq I }\widehat{f_m}(J) h_J\right)\\
&=& \sum_{\vec{\alpha} \in U_{m-1}} \sum_{I\in\D} P_I^{\vec{\alpha}}(f_1,f_2, \ldots, f_{m-1}) \La f_m \Ra_I + \sum_{\vec{\alpha} \in U_{m-1}} \sum_{I\in\D} P_I^{\vec{\alpha}}(f_1,f_2, \ldots, f_{m-1}) \widehat{f_m}(I) h_I\\
&& \quad + \sum_{\vec{\alpha} \in U_{m-1}} \sum_{I\in\D} P_I^{\vec{\alpha}}(f_1,f_2, \ldots, f_{m-1})\left(\sum_{J\subsetneq I }\widehat{f_m}(J) h_J \right)\\
&=& \sum_{\vec{\alpha} \in U_{m-1}} \sum_{I\in\D} P_I^{(\vec{\alpha},1)}(f_1,f_2, \ldots, f_m) + \sum_{\vec{\alpha} \in U_{m-1}} \sum_{I\in\D} P_I^{(\vec{\alpha},0)}(f_1,f_2, \ldots, f_m)\\
&& \quad + \sum_J\widehat{f_m}(J) h_J  \left(\sum_{\vec{\alpha} \in U_{m-1}} \sum_{J\subsetneq I} P_I^{\vec{\alpha}}(f_1,f_2, \ldots, f_{m-1})\right)\\
&=& \sum_{\vec{\alpha} \in U_{m-1}} \sum_{I\in\D} P_I^{(\vec{\alpha},1)}(f_1,f_2, \ldots, f_m) + \sum_{\vec{\alpha} \in U_{m-1}} \sum_{I\in\D} P_I^{(\vec{\alpha},0)}(f_1,f_2, \ldots, f_m)\\
&& \quad + \sum_J\widehat{f_m}(J) h_J \La f_1 \Ra_J \ldots\La f_{m-1} \Ra_J\\
&=& \sum_{\vec{\alpha} \in U_{m-1}} \sum_{I\in\D} P_I^{(\vec{\alpha},1)}(f_1,f_2, \ldots, f_m) + \sum_{\vec{\alpha} \in U_{m-1}} \sum_{I\in\D} P_I^{(\vec{\alpha},0)}(f_1,f_2, \ldots, f_m)\\
&& \quad + P^{(1,\ldots,1,0)}(f_1,f_2, \ldots, f_m)\\
&=& \sum_{\vec{\alpha} \in U_m} P^{\vec{\alpha}}(f_1,f_2, \ldots, f_m).
\end{eqnarray*}
Here the last equality follows from $(\ref{IndexSetUm})$.
\end{proof}

\subsection{Multilinear dyadic paraproducts}

\noindent
On the basis of the decomposition of pointwise product $\prod_{j=1}^m f_j$ we now define multi-linear dyadic paraproduct operators, and study their boundedness properties.\\
\begin{dfn} 
For $m \geq 2$ and $\vec{\alpha} = (\alpha_1, \alpha_2, \ldots, \alpha_m) \in \{0,1\}^m$, we define \textit{multi-linear dyadic paraproduct operators} by 
$$ P^{\vec{\alpha}}(f_1,f_2,\ldots,f_m) = \sum_{I\in\mathcal{D}} \prod_{j=1}^m f_j(I,\alpha_j) h_I^{\sigma(\vec{\alpha})} $$
where $f_i(I,0) = \langle f_i, h_I \rangle$, $f_i(I,1) = \langle f_i \rangle_I$ and $\sigma(\vec{\alpha}) = \#\{i: \alpha_i = 0\}.$\\
\end{dfn}
\noindent
Observe that if $\vec{\beta} = (\beta_1, \beta_2, \ldots,\beta_m)$ is some permutation of $\vec{\alpha} = (\alpha_1, \alpha_2, \ldots, \alpha_m)$ and $(g_1, g_2, \ldots, g_m)$ is the corresponding permutation of $(f_1, f_2, \ldots, f_m)$, then
$$P^{\vec{\alpha}} (f_1, f_2, \ldots, f_m) = P^{\vec{\beta}} (g_1, g_2, \ldots, g_m).$$
\noindent
Also note that $P^{(1,0)}$ and $P^{(0,1)}$ are the standard bilinear paraproduct operators:
$$ P^{(0,1)}(f_1,f_2) = \sum_{I\in\mathcal{D}} \langle f_1, h_I \rangle \langle f_2 \rangle_I h_I = P(f_1,f_2)$$
$$ P^{(1,0)}(f_1,f_2) = \sum_{I\in\mathcal{D}} \langle f_1 \rangle_I \langle f_2, h_I \rangle h_I = P(f_1,f_2).$$

\noindent
In terms of paraproducts, the decomposition of point-wise product $\displaystyle\prod_{j=1}^m f_j$ we  obtained in the previous section takes the form
$$\displaystyle\prod_{j=1}^m f_j = \displaystyle \sum_{\substack {\vec{\alpha} \in \{0,1\}^m\\ \vec{\alpha} \neq (1,1,\ldots,1)}} P^{\vec{\alpha}}(f_1,f_2,\ldots,f_m).$$

\noindent 
\begin{dfn}
For a given function $b$ and $\vec{\alpha} = (\alpha_1, \alpha_2, \ldots, \alpha_m) \in \{0,1\}^m$, we define the paraproduct operators $\pi_b^{\vec{\alpha}}$ by
$$\pi_b^{\vec{\alpha}}(f_1, f_2, \ldots, f_m) = P^{(0,\vec{\alpha})}(b,f_1, f_2, \ldots, f_m) = \sum_{I \in \mathcal{D}} \La b , h_I \Ra \prod_{j=1}^m f_j(I, \alpha_j) \; h_I^{1+\sigma(\vec{\alpha})}$$
where $(0,\vec{\alpha}) = (0,\alpha_1,\ldots, \alpha_m) \in \{0,1\}^{m+1}.$\\
\end{dfn}
\noindent Note that $$\pi_b^1(f) = P^{(0,1)}(b,f) = \sum_{I \in \mathcal{D}} b(I,0) f(I,1) h_I = \sum_{I \in \mathcal{D}} \La b, h_I \Ra \La f \Ra_I h_I = \pi_b(f).$$

\noindent
The rest of this section is devoted to the boundedness properties of these multilinear paraproduct operators $P^{\vec{\alpha}}$ and $\pi_b^{\vec{\alpha}}.$\\

\noindent
\begin{lm}\label{MPPTh1} Let $1 <p_1,p_2,\ldots,p_m, r < \infty$ and \,$\sum_{j=1}^m \frac{1}{p_j} = \frac{1}{r}$. Then for $\vec{\alpha} = (\alpha_1, \alpha_2, \ldots, \alpha_m) \in U_m$, the operators $P^{\vec{\alpha}}$ map $L^{p_1} \times \cdots\times L^{p_m} \rightarrow L^{r}$ with estimates of the form:
$$\norm{P^{\vec{\alpha}}(f_1,f_2,\ldots,f_m)}_r \lesssim \prod_{j=1}^m\norm{f_j}_{p_j}$$
\end{lm}

\noindent
\begin{proof} First we observe that, if $x\in I \in \D,$ then $$\abs{\La f \Ra_I} \leq \La \abs{f}\Ra_I  \leq Mf(x)$$ and that
\begin{eqnarray*}
\frac{\left\vert \La f , h_I \Ra \right\vert}{\sqrt{\abs{I}}} & = & \frac{1}{\sqrt{\abs{I}}} \left\vert \int_\R f h_I  \right\vert\\
& = & \frac{1}{\abs{I}} \left\vert \int_\R f \mathsf{1}_{I_+}  - \int_\R f \mathsf{1}_{I_-}  \right\vert\\
& = & \frac{1}{\abs{I}} \left(\int_{I_+} \abs{f}   + \int_{I_-} \abs{f}   \right)\\
& \leq & \frac{1}{\abs{I}}  \int_{I} \abs{f}   \\
& \leq & Mf(x).
\end{eqnarray*}

\noindent
\textbf{Case I:} $\sigma(\vec{\alpha}) = 1.$\\
Let $\alpha_{j_0} = 0.$ Then
\begin{eqnarray*}
\displaystyle P^{\vec{\alpha}}(f_1,f_2,\ldots,f_m) &=& \sum_{I\in\mathcal{D}} \prod_{j=1}^m f_j(I,\alpha_j) h_I^{\sigma(\vec{\alpha})}\\
&=& \sum_{I\in\mathcal{D}} \left(\prod_{\substack{j = 1\\j\neq j_0}}^m \La f_j \Ra_I\right) \La f_{j_0}, h_I\Ra h_I.
\end{eqnarray*}
\noindent
Using square function estimates, we obtain
\begin{eqnarray*}
\left\Vert P^{\vec{\alpha}}(f_1, f_2, \ldots, f_m) \right\Vert_r &\lesssim& \left\Vert\left(\sum_{I\in \D} \prod_{\substack{j = 1\\j\neq j_0}}^m \left\vert\La f_j\Ra_I \right\vert^2 \abs{\La f_{j_0}, h_I \Ra}^2 \frac{\mathsf{1}_I}{\abs{I}}\right)^{1/2}\right\Vert_r\\
&\leq& \left\Vert\left(\prod_{\substack{j = 1\\j\neq j_0}}^m Mf_j \right) \left(\sum_{I\in \D}  \abs{\La f_{j_0}, h_I \Ra}^2 \frac{\mathsf{1}_I}{\abs{I}}\right)^{1/2}\right\Vert_r\\
&= & \left\Vert\left(\prod_{\substack{j = 1 \\j\neq j_0}}^m Mf_j \right) (Sf_{j_0}) \right\Vert_r\\
&\leq& \prod_{\substack{j = 1\\j\neq j_0}}^m \norm{Mf_j}_{p_j} \norm{Sf_{j_0}}_{j_0}\\
&\lesssim& \prod_{j=1}^m \norm{f_j}_{p_j},
\end{eqnarray*}
where we have used H$\ddot{\text{o}}$lder inequality, and the boundedness of maximal and square function operators  to obtain the last two inequalities.\\

\noindent
\textbf{Case II:} $\sigma(\vec{\alpha}) > 1.$\\
Choose $j'$ and $j''$ such that $\alpha_{j'} = \alpha_{j''} = 0.$ Then 
\begin{eqnarray*}
\left\vert P^{(0,0,\ldots,0)}(f_1,f_2,\ldots,f_m)(x) \right\vert & = &  \left\vert \sum_{I\in \D}\left(\prod_{j:\alpha_j = 1} \La f_j \Ra_I\right) \left( \prod_{\substack{j:\alpha_j = 0 \\j \neq j', \,j''}} \frac{\La f_j, h_I\Ra}{\sqrt{\abs{I}}} \right) \La f_{j'}, h_I\Ra \La f_{j''}, h_I\Ra \frac{\mathsf{1}_I(x)}{\abs{I}}  \right\vert \\
& \leq & \left(\prod_{ j: j \neq j',\,j''} Mf_j(x) \right) \left( \sum_{I\in \D} \abs{\La f_{j'}, h_I\Ra} \abs{\La f_{j''}, h_I\Ra} \frac{\mathsf{1}_I(x)}{\abs{I}}  \right). 
\end{eqnarray*}
By Cauchy-Schwarz inequality
\begin{eqnarray}
 \nonumber && \sum_{I\in\D} \left\vert\La f_{j'},h_I\Ra \right\vert \, \left\vert\La f_{j''},h_I \Ra\right\vert \frac{\mathsf{1}_I(x)}{\abs{I}}\\ 
\label{eq:sf}&\leq&  \left(\sum_{I\in\D}\abs{\La f_{j'},h_I\Ra }^2 \frac{\mathsf{1}_I(x)}{\abs{I}}\right)^{\frac{1}{2}} \left(\sum_{I\in\D}\abs{\La f_{j''},h_I\Ra }^2 \frac{\mathsf{1}_I(x)}{\abs{I}}\right)^{\frac{1}{2}}\\
 \nonumber &=& Sf_{j'}(x)\, Sf_{j''}(x).
\end{eqnarray}
Therefore, 
\begin{eqnarray*}
\left\vert P^{(0,0,\ldots,0)}(f_1,f_2,\ldots,f_m)(x) \right\vert 
& \leq & \left(\prod_{ j: j \neq j',\,j''} Mf_j(x) \right) Sf_{j'}(x)\, Sf_{j''}(x). 
\end{eqnarray*}
\noindent
Now using generalized H$\ddot{\text{o}}$lder's inequality and the boundedness properties of the maximal and square functions, we get \begin{eqnarray*} \left\Vert P^{(0,0,\ldots,0)}(f_1,f_2,\ldots,f_m) \right\Vert_r
 &\leq& \left(\prod_{ j: j \neq j',\,j''} \norm{Mf_j}_{p_j}\right) \norm{Sf_{j'}}_{p_{j'}}\, \norm{Sf_{j''}}_{p_{j''}}\\
& \lesssim &\prod_{j=1}^m\norm{f_j}_{p_j}.
\end{eqnarray*}

\end{proof}

\noindent
\begin{lm} \label{MPPTh2}
Let $\vec{\alpha} = (\alpha_1, \ldots, \alpha_m) \in \{0,1\}^m$ and $1 <p_1, \ldots,p_m,r < \infty$ with $\sum_{j=1}^m \frac{1}{p_j} = \frac{1}{r}.$
\begin{enumerate}[label = $(\alph*)$]
\item  For $\sigma(\vec{\alpha}) \leq 1,$ $\pi_b^{\vec{\alpha}}$ is a bounded operator from $L^{p_1} \times \cdots \times L^{p_m}$ to $L^{r}$ if and only if \qquad $b \in BMO^d.$\\
\item  For $\sigma(\vec{\alpha}) > 1,$ $\pi_b^{\vec{\alpha}}$ is a bounded operator from $L^{p_1} \times \cdots \times L^{p_m}$ to $L^{r}$ if and only if $\displaystyle\sup_{I\in \D} \frac{\abs{\La b,h_I\Ra}}{\sqrt{\abs{I}}} < \infty.$
\end{enumerate}
\end{lm}
\begin{proof}
$(a)$ We prove this part first for $\sigma(\vec{\alpha}) = 0,$ that is, for $\alpha_1 = \cdots = \alpha_m = 1.$\\

\noindent
Assume that $b \in BMO^d.$ Then for $(f_1, \ldots,f_m) \in L^{p_1} \times \cdots \times L^{p_m},$ we have
\begin{eqnarray*}
\pi_b^{\vec{\alpha}} (f_1,\ldots,f_m) 
&=& P^{(0,\vec{\alpha})}(b,f_1,\ldots,f_m) \\
&=& \sum_{I\in\mathcal{D}} \La b, h_I \Ra \prod_{j=1}^m \La f_j \Ra_I h_I\\
&=& \sum_{I\in\mathcal{D}} \La \pi_b(f_1), h_I \Ra \prod_{j=2}^m \La f_j \Ra_I  h_I\\
&=& P^{(0, \alpha_2, \ldots,\alpha_m)} \left(\pi_b(f_1), f_2, \ldots, f_m \right).
\end{eqnarray*}
Since $b \in BMO^d$ and $f_1 \in L^{p_1}$ with $p_1 > 1,$ we have $\norm {\pi_b(f_1)}_{p_1} \lesssim \norm{b}_{BMO^d} \norm{f_1}_{p_1}.$ So,
\begin{eqnarray*}
\norm{\pi_b^{\vec{\alpha}} (f_1,\ldots,f_m)}_r
&=& \norm{P^{(0,\alpha_2, \ldots, \alpha_m)}\left(\pi_b(f_1), f_2, \ldots, f_m \right)}_r \\
&\lesssim & \norm {\pi_b(f_1)}_{p_1} \prod_{j=2}^m\norm{f_j}_{p_j}\\
&\lesssim & \norm{b}_{BMO^d}\prod_{j=1}^m\norm{f_j}_{p_j},
\end{eqnarray*}
where the first inequality follows from Theorem \ref{MPPTh1}.\\

\noindent
Conversely, assume that $\pi_b^{(1,\ldots,1)}: L^{p_1} \times \cdots \times L^{p_m}\rightarrow L^{r}$ is bounded. Then for $f_i = \abs{J}^{-\frac{1}{p_i}}\mathsf{1}_J(x)$ with $J \in \D,$
$$\left\Vert \pi_b^{(1,1,\ldots,1)}(f_1,f_2,\ldots,f_m)\right\Vert_r \leq \left \Vert \pi_b^{(1,1,\ldots,1)} \right \Vert_{L^{p_1}\times \cdots\times L^{p_m} \rightarrow L^r}, $$
since $\norm{f_i}_{p_i} = 1$ for all $ 1 \leq i \leq m.$ For such $f_i,$
\begin{eqnarray*}
\left\Vert \pi_b^{(1,1,\ldots,1)}(f_1,f_2,\ldots,f_m)\right\Vert_r &=& \left\Vert \abs{J}^{-\left(\frac{1}{p_1}+\frac{1}{p_2}+\cdots+\frac{1}{p_m}\right)}\;\pi_b^{(1,1,\ldots,1)}(\mathsf{1}_J,\mathsf{1}_J,\ldots,\mathsf{1}_J)\right\Vert_r \\
&=& \abs{J}^{-\frac{1}{r}}\left\Vert \sum_{I\in\D}\widehat{b}(I) \La \mathsf{1}_J\Ra_I^m h_I  \right\Vert_r.
\end{eqnarray*}
\noindent
Taking $\epsilon_I = 1$ if $I\subseteq J$ and $\epsilon_I = 0$ otherwise, we observe that
\begin{eqnarray*}
\left\Vert \sum_{J\supseteq I \in \D}\widehat{b}(I) h_I  \right\Vert_r 
&=& \left\Vert \sum_{J\supseteq I \in \D}\widehat{b}(I) \La \mathsf{1}_J\Ra_I^m h_I  \right\Vert_r \\
&=& \left\Vert \sum_{I\in\D}\epsilon_I\widehat{b}(I) \La \mathsf{1}_J\Ra_I^m h_I  \right\Vert_r\\
&\lesssim& \left\Vert \sum_{I\in\D}\widehat{b}(I) \La \mathsf{1}_J\Ra_I^m h_I  \right\Vert_r,
\end{eqnarray*}
where the last inequality follows from the boundedness of Haar multiplier $T_\epsilon$ on $L^r.$
Thus, we have
\begin{eqnarray*}
\sup_{J\in\D}\abs{J}^{-1/r}\left\Vert \sum_{J\supseteq I \in \D}\widehat{b}(I) h_I  \right\Vert_r 
&\lesssim & \sup_{J\in\D}\abs{J}^{-1/r} \left\Vert \sum_{I\in\D}\widehat{b}(I) \La \mathsf{1}_J\Ra_I^m h_I  \right\Vert_r\\
&\lesssim& \left \Vert \pi_b^{(1,1,\ldots,1)} \right \Vert_{L^{p_1}\times \cdots\times L^{p_m} \rightarrow L^r},
\end{eqnarray*}
proving that $b\in BMO^d.$\\

\noindent
Now the proof for $\sigma(\vec{\alpha}) = 1$ follows from the simple observation that $\pi_b^{\vec{\alpha}}$ is a transpose of $\pi_b^{(1,\ldots,1)}$. For example, if $\sigma(\vec{\alpha}) = 1$ with $\alpha_1 = 0$ and $\alpha_2 = \cdots = \alpha_m =1$ and if $r'$ is the conjugate exponent of $r,$ then for $g \in L^{r'}$
\begin{eqnarray*}
\left\La \pi_b^{\vec{\alpha}}(f_1,\ldots,f_m), g \right\Ra
&=& \left\La \sum_{I \in \D} \La b, h_I \Ra \La f_1, h_I \Ra \prod_{j=2}^m \La f_j \Ra_I h_I^2, g \right\Ra\\
&=& \sum_{I \in \D} \La b, h_I \Ra \La f_1, h_I \Ra \prod_{j=2}^m \La f_j \Ra_I \La g, h_I^2\Ra\\
&=& \sum_{I \in \D} \La b, h_I \Ra \La f_1, h_I \Ra \prod_{j=1}^m \La f_j \Ra_I \La g\Ra_I\\
&=& \left\La \sum_{I \in \D} \La b, h_I \Ra \La g \Ra_I \prod_{j=1}^m \La f_j \Ra_I h_I, f_1 \right\Ra\\
&=& \left\La \pi_b^{(1, \ldots, 1)}(g,f_2,\ldots,f_m), f_1 \right\Ra.
\end{eqnarray*}

\noindent
$(b)$ Assume that $ \norm {b}_*\equiv \displaystyle\sup_{I\in \D} \frac{\abs{\La b,h_I\Ra}}{\sqrt{\abs{I}}} < \infty.$ For $m =2$ we have
\begin{eqnarray*}
\displaystyle \int_\R \left\vert \pi_b^{(0,0)}(f_1,f_2) \right\vert^r dx 
&=& \displaystyle\int_\R \left\vert \sum_{I\in\D}\La b,h_I\Ra \La f_1,h_I\Ra \La f_2,h_I \Ra h_I^3(x) \right\vert^r dx \\
&\leq & \int_\R \left( \sum_{I\in\D}\abs{\La b,h_I\Ra}\, \abs{\La f_1,h_I\Ra }\, \abs{\La f_2,h_I \Ra} \frac{\mathsf{1}_I(x)}{\abs{I}^{3/2}} \right)^r dx  \\
&\leq & \int_\R \left( \sup_{I\in\D} \frac{\abs{\La b,h_I\Ra}}{\sqrt{\abs{I}}}\sum_{I\in\D} \abs{\La f_1,h_I\Ra }\, \abs{\La f_2,h_I \Ra} \frac{\mathsf{1}_I(x)}{\abs{I}} \right)^r dx\\
&=& \norm{b}_*^r  \int_\R \left(\sum_{I\in\D} \abs{\La f_1,h_I\Ra }\, \abs{\La f_2,h_I \Ra} \frac{\mathsf{1}_I(x)}{\abs{I}} \right)^r dx.
\end{eqnarray*}
Using \eqref{eq:sf} and H$\ddot{\text{o}}$lder's inequality we obtain
\begin{eqnarray*}
\displaystyle \int_\R \left\vert \pi_b^{(0,0)}(f_1,f_2) \right\vert^r dx 
&\leq& \norm{b}_*^r \int_\R (Sf_1)^r(x)\,(Sf_2)^r(x)\,dx\\
&\leq& \norm{b}_*^r \left(\int_\R \left\{(Sf_1)^r(x)\right\}^{p_1/r}\,dx\right)^{r/p_1}\left(\int_\R \left\{(Sf_2)^r(x)\right\}^{p_2/r}\,dx\right)^{r/p_2}\\
&\leq& \norm{b}_*^r \norm{Sf_1}_{p_1}^r \norm{Sf_2}_{p_2}^r\\
&\lesssim & \norm{b}_*^r \norm{f_1}_{p_1}^r \norm{f_2}_{p_2}^r.
\end{eqnarray*}
\noindent
Thus we have,
$$ \norm{\pi_b^{(0,0)}(f_1,f_2)}_r \lesssim \norm{b}_* \norm{f_1}_{p_1} \norm{f_2}_{p_2}.$$
\noindent
Observe that $$\pi_b^{(0,0)}(f_1,f_2)(I,0) = \La \pi_b^{(0,0)}(f_1,f_2), h_I \Ra = \frac{1}{\abs{I}}\La b,h_I\Ra \La f_1,h_I\Ra \La f_2,h_I \Ra.$$
\noindent
Now consider $m > 2$ and let $\sigma(\vec{\alpha})>1$. Without loss of generality we may assume that $\alpha_1 = \alpha_2 = 0.$ Then\\
\begin{eqnarray*}
\norm{\pi_b^{\vec{\alpha}} (f_1,f_2, \ldots,f_m)}_r & = & \left\Vert \sum_{I\in\D} \La b,h_I\Ra \La f_1,h_I\Ra \La f_2,h_I\Ra\prod_{j=3}^m f_j(I, \alpha_j) h_I^{1+\sigma(\vec{\alpha})}\right\Vert_r\\
& = & \left\Vert \sum_{I\in\D} \frac{1}{\abs{I}}\La b,h_I\Ra \La f_1,h_I\Ra \La f_2,h_I\Ra \prod_{j=3}^m f_j(I, \alpha_j) h_I^{\sigma(\vec{\alpha})-1}\right\Vert_r\\
&=& \left\Vert \sum_{I\in\D} \La \pi_b^{(0,0)}(f_1,f_2), h_I \Ra \prod_{j=3}^m f_j(I, \alpha_j)  h_I^{\sigma(\vec{\alpha})-1}\right\Vert_r\\
&=& \left\Vert P^{\vec{\beta}}(\pi_b^{(0,0)}(f_1,f_2),f_3,\ldots,f_m) \right\Vert_r\\
&\lesssim& \norm{\pi_b^{(0,0)}(f_1,f_2)}_{q} \prod_{j=3}^m \norm{f_j}_{p_j}\\
&\lesssim& \norm{b}_* \prod_{j=1}^m\norm{f_j}_{p_j}\\
\end{eqnarray*} 
where $\vec{\beta} =(0,\alpha_3,\ldots,\alpha_m) \in \{0,1\}^{m-1}$ and $\pi_b^{(0,0)}(f_1,f_2) \in L^q$ with $\frac{1}{p_1}+\frac{1}{p_2}=\frac{1}{q}, q>r>1.$\\

\noindent
Conversely, assume that $\pi_b^{\vec{\alpha}}: L^{p_1} \times \cdots \times L^{p_m}\rightarrow L^{r}$ is bounded and that $\sigma(\vec{\alpha}) > 1.$ Choose any $J \in \D$, and take $f_j = \abs{J}^{\frac{1}{2} - \frac{1}{p_j}} h_J$ if $\alpha_j = 0,$ and $f_j = \abs{J}^{-\frac{1}{p_j} } \mathsf{1}_J$ if $\alpha_j = 1$ so that $\norm{f_j}_{p_j} = 1.$ Then
$$ \left \Vert \pi_b^{\vec{\alpha}} (f_1, \ldots, f_m) \right \Vert_r \leq \left \Vert \pi_b^{\vec{\alpha}} \right \Vert _{L^{p_1} \times \cdots \times L^{p_m}}. $$
We also have
\begin{eqnarray*}
\left \Vert \pi_b^{\vec{\alpha}} (f_1, \ldots, f_m) \right \Vert_r &=& \left \Vert \abs{J}^{\frac{\sigma{(\vec{\alpha})}}{2} - \sum_{j=1}^m \frac{1}{p_j}} \La b, h_J \Ra h_J^{1+\sigma(\vec{\alpha})} \right \Vert_r \\
&=& \abs{J}^{\frac{\sigma{(\vec{\alpha})}}{2} - \frac{1}{r}} \abs{\La b, h_J \Ra} \left \Vert h_J^{1+\sigma(\vec{\alpha})} \right \Vert_r \\
&=& \abs{J}^{\frac{\sigma(\vec{\alpha})}{2} - \frac{1}{r}} \abs{\La b, h_J \Ra} \abs{J}^{-\frac{1+\sigma(\vec{\alpha})}{2}}\left \Vert \mathsf{1}_J \right \Vert_r \\
&=& \abs{J}^{\frac{\sigma(\vec{\alpha})}{2} - \frac{1}{r}} \abs{\La b, h_J \Ra} \abs{J}^{-\frac{1+\sigma(\vec{\alpha})}{2}}\abs{J}^{\frac{1}{r}}\\
&=& \frac{\abs{\La b, h_J \Ra}}{\sqrt{\abs{J}}}.
\end{eqnarray*}

\noindent
Thus  $ \frac{\abs{\La b, h_J \Ra}}{\sqrt{\abs{J}}} \leq \left \Vert \pi_b^{\vec{\alpha}}\right \Vert _{L^{p_1} \times \cdots \times L^{p_m}}.$ Since it is true for any $J \in D,$ we have $$ \displaystyle \sup_{J \in \D} \frac{\abs{\La b, h_J \Ra}}{\sqrt{\abs{J}}} \leq \left \Vert \pi_b^{\vec{\alpha}} \right \Vert _{L^{p_1} \times \cdots \times L^{p_m}} < \infty,$$ as desired.\\
\end{proof}

\noindent
Now that we have obtained strong type $L^{p_1} \times\cdots\times L^{p_m} \rightarrow L^r$ boundedness estimates for the paraproduct operators $P^{\vec{\alpha}}$ and $\pi_b^{\vec{\alpha}}$ when $1 < p_1, p_2, \ldots, p_m,r < \infty$ and $\sum_{j=1}^m \frac{1}{p_j} = \frac{1}{r}$, we are interested to investigate estimates corresponding to $\frac{1}{m} \leq r < \infty$. We will prove in Lemma $\ref{MPPL}$ that we obtain weak type estimates if one or more $p_i$'s are equal to 1. In particular, we obtain $L^{1} \times\cdots\times L^{1} \rightarrow L^{\frac{1}{m},\infty}$ estimates for those operators. Then it follows from multilinear interpolation that the paraproduct operators are strongly bounded from $L^{p_1} \times \cdots \times L^{p_m}$  to $L^r$ for $1 < p_1, p_2, \ldots, p_m < \infty$ and $\sum_{j=1}^m \frac{1}{p_j} = \frac{1}{r},$ even if $\frac{1}{m} < r \leq 1.$\\

\noindent
We first prove the following general lemma, which when applied to the operators $P^{\vec{\alpha}}$ and $\pi_b^{\vec{\alpha}}$ gives aforementioned weak type estimates.\\ 

\noindent
\begin{lm}\label{WBL} Let $T$ be a multilinear operator that is bounded from the product of Lebesgue spaces $L^{p_1} \times \cdots \times L^{p_m}$ to  $L^{r,\infty}$ for some $1 < p_1, p_2, \ldots, p_m < \infty$ with 
$$\sum_{j=1}^m \frac{1}{p_j} = \frac{1}{r}.$$
Suppose that for every $I \in \mathcal{D}$, $T(f_1, \ldots, f_m)$ is supported in $I$ if $f_i = h_I$ for some $i \in \{1, 2, \ldots, m\}$. Then $T$ is bounded from $L^1 \times \cdots \times L^1 \times L^{p_{k+1}} \times \cdots\times L^{p_m}  \rightarrow  L^{\frac{q_k}{q_k + 1},\infty}$ for each $k = 1, 2, \ldots,m,$ where $q_k$ is given by  
$$\frac{1}{q_k} = (k-1) + \frac{1}{p_{k+1}} + \cdots+\frac{1}{p_{m}}.$$
In particular, $T$ is bounded from $L^{1} \times \cdots \times L^{1}$ to $L^{\frac{1}{m},\infty}$.
\end{lm}

\noindent
\begin{proof} We first prove that $T$ is bounded from $L^{1} \times L^{p_2}\times \cdots \times L^{p_m} $ to $ L^{\frac{q_1}{q_{1}+1},\infty}.$\\
Let $\lambda > 0$ be given. We have to show that
$$\abs {\{ x: \abs {T(f_1, f_2, \ldots,f_m)(x)} > \lambda \}} \lesssim \left(\frac{\norm{f_1}_1 \prod_{j=2}^m\norm{f_j}_{p_j}}{\lambda}\right)^{\frac{q_1}{1+q_1}}$$
for all $(f_1, f_2, \ldots,f_m) \in L^{1} \times L^{p_2} \cdots \times L^{p_m}$.\\
Without loss of generality, we assume $\norm{f_1}_{1} = \norm{f_2}_{p_2}= \cdots = \norm{f_m}_{p_m} =1,$ and prove that
$$\abs {\{ x: \abs {T(f_1, f_2, \ldots,f_m)(x)} > \lambda \}} \lesssim \lambda^{-\frac{q_1}{1+q_1}}.$$
For this, we apply Calder$\acute{\text{o}}$n-Zygmund decomposition to the function $f_1$ at height $\lambda^{\frac{q_1}{q_{1}+1}}$ to obtain `good' and `bad' functions $g_1$ and $b_1$, and a sequence $\{I_{1,j}\}$ of disjoint dyadic intervals such that
$$ f_1 = g_1 + b_1, \;\;\; \norm{g_1}_{p_1} \leq \left(2 \lambda^{\frac{q_1}{q_{1}+1}}\right)^{
'} \norm{f_1}_1^{1/p_1} = \left(2 \lambda^{\frac{q_1}{q_{1}+1}}\right)^{\frac{p_1-1}{p_1}}\;\;\; \text{ and } \;\;\; b_1 = \sum_j b_{1,j},$$
where $$\text{supp}(b_{1,j}) \subseteq I_{1,j},\;\; \; \int_{I_{1,j}} b_{1,j} dx = 0,\;\; \text{ and } \;\; \sum_j\abs{I_{1,j}} \leq {\lambda}^{-\frac{q_1}{q_1+1}}\norm{f_1}_1 = {\lambda}^{-\frac{q_1}{q_1+1}}.$$\\
Multilinearity of $T$ implies that
$$\left|\left\{x:|T(f_1,\ldots,f_m)(x)| > \lambda \right\} \right|$$
$$ \leq \left|\left\{x:|T(g_1, f_2, \ldots,f_m)(x)| > \frac{\lambda}{2}\right\}\right| \; + \; \left|\left \{x:|T(b_1, f_2, \ldots,f_m)(x)| > \frac{\lambda}{2} \right \}\right|.$$
Since $g_1 \in L^{p_1}$ and $T$ is bounded from $L^{p_1} \times \cdots \times L^{p_m}$ to $L^{r,\infty}$, we have
\begin{eqnarray*}
\abs {\{ x: \abs {T(g_1, f_2, \ldots,f_m)(x)} > \lambda/2 \}}
 & \lesssim & \left(\frac{2\norm{g_1}_{p_1} \displaystyle\prod_{j=2}^m f_j(J,\alpha_j)}{\lambda}\right)^r\\
& \leq & \left(\frac{2\left(2 \lambda^{\frac{q_1}{q_{1}+1}}\right)^{\frac{p_1 -1}{p_1}}} {\lambda}\right)^r\\
& \lesssim & \lambda^{r\left(\frac{q_1(p_1 -1)}{p_1(q_{1}+1)} -1 \right)}
\end{eqnarray*}
Now,  $\frac{1}{r} =  \sum_{j=1}^m \frac{1}{p_j} = \frac{1}{p_1} + \frac{1}{q_1}$ implies that $r = \frac{p_1 q_1} {p_1+q_1}.$ So,\\
\begin{eqnarray*} r\left(\frac{q_1(p_1 -1)}{p_1(q_{1}+1)} -1\right) &=& \frac{p_1 q_1}{(p_1+q_1)}\left(\frac{p_1q_1 - q_1 - p_1q_1 - p_1}{p_1(q_{1}+1)}\right)\\
 &=& \frac{p_1 q_1}{(p_1+q_1)}\frac{(-p_1 - q_1)}{p_1(q_{1}+1)}\\
 &=& -\frac{q_1}{q_1+1}. 
\end{eqnarray*}

\noindent
Thus we have:
$$\abs {\{ x: \abs {T(g_1, f_2, \ldots,f_m)(x)} > \lambda/2 \}} \lesssim \lambda^{-\frac{q_1}{1+q_1}}.$$

\noindent
From the properties of `bad' function $b_1$ we deduce that $\La b_1, h_I \Ra \neq 0$ only if $I \subseteq I_{1,j}$ for some $j$. The hypothesis of the lemma on the support of $T(f_1, \ldots, f_m)$ then implies that 
$$ \text{supp}\left(T(b_1,f_2, \ldots, f_m)\right) \subseteq \cup_j I_{1,j}.$$ 
Thus, $$ \left|\left \{x:|T(b_1, f_2, \ldots,f_m)(x)| > \frac{\lambda}{2} \right \}\right| \leq \left| \cup_j I_{1,j} \right| \leq \lambda^{-\frac{q_1}{1+q_1}}.$$

\noindent
Combining these estimates corresponding to $g_1$ and $b_1$, we have the desired estimate
$$\abs {\{ x: \abs {T(f_1, f_2, \ldots,f_m)(x)} > \lambda \}} \lesssim \lambda^{-\frac{q_1}{1+q_1}}.$$

\noindent
Now beginning with the $L^{1} \times L^{p_2}\times \cdots \times L^{p_m} \rightarrow  L^{\frac{q_1}{q_{1}+1},\infty}$ estimate, we use the same argument to lower the second exponent to 1 proving that $T$ is bounded from $L^{1} \times L^{1}\times L^{p_3} \times \cdots \times L^{p_m} $ to $ L^{\frac{q_2}{q_{2}+1},\infty}, $ where $q_2$ is given by $\frac{1}{q_2} = 1 + \frac{1}{p_{3}} + \cdots +\frac{1}{p_{m}}.$\\

\noindent
We continue the same process until we obtain $L^{1} \times L^{1}\times \cdots \times L^{1} \rightarrow  L^{\frac{q_m}{q_{m}+1},\infty}$ boundedness of $T$ with $\frac{1}{q_m} = 1+1+ \cdots+1 \; (m-1 \text{ terms}) = m-1.$ This completes the proof since $\frac{q_m}{q_m+1} = \frac{1}{m}.$ \end{proof}
\noindent 
\begin{lm} \label{MPPL}
 Let $ \vec{\alpha} = (\alpha_1,\alpha_2,\ldots,\alpha_m) \in \{0,1\}^m, 1 \leq p_1, p_2, \ldots, p_m < \infty$ and $\sum_{j=1}^m \frac{1}{p_j} = \frac{1}{r}.$ Then
\begin{enumerate}[label = $(\alph*)$]
\item For $\vec{\alpha}  \neq (1,1,\ldots,1),$  $P^{\vec{\alpha}}$ is bounded from $L^{p_1} \times \cdots \times L^{p_m}$  to $L^{r,\infty}.$ 
\item If $b \in BMO^d$ and $\sigma(\vec{\alpha}) \leq 1, \,\pi_b^{\vec{\alpha}}$ is bounded from $L^{p_1} \times \cdots \times L^{p_m}$  to  $L^{r,\infty}.$ 
\item If $\displaystyle\sup_{I\in \D} \frac{\abs{\La b,h_I\Ra}}{\sqrt{\abs{I}}} < \infty$ and $\sigma(\vec{\alpha}) > 1, \,\pi_b^{\vec{\alpha}}$ is bounded from $L^{p_1} \times \cdots \times L^{p_m}$  to  $L^{r,\infty}.$ 
\end{enumerate}
\end{lm}
\begin{proof} By orthogonality of Haar functions, $h_I(J,0) = \La h_I, h_J \Ra = 0 $ for any two distinct dyadic intervals $I$ and $J.$ The Haar functions have mean value 0,  so it is easy to see that
$$\La h_I \Ra_J \neq 0 \text{ only if } J \subsetneq I$$
since any two dyadic intervals are either disjoint or one is contained in the other.\\

\noindent
Consequently, if some $f_i = h_I,$ then
 $$P^{\vec{\alpha}}(f_1, f_2,\ldots,f_m) = \sum_{J\subseteq I}\prod_{j=1}^m f_j(J,\alpha_j) h_J^{\sigma(\vec{\alpha})}$$
 and, $$\pi_b^{\vec{\alpha}}(f_1, f_2,\ldots,f_m) = \sum_{J\subseteq I}\La b,h_J \Ra \prod_{j=1}^m f_j(J,\alpha_j) h_J^{1+ \sigma(\vec{\alpha})},$$ 
which are both supported in $I.$ Since the paraproducts are strongly (and hence weakly) bounded from $L^{p_1} \times \cdots \times L^{p_m}\rightarrow L^r$, the proof follows immediately from Lemma $\ref{WBL}.$ \end{proof} 

\noindent
Combining the results of Lemmas \ref{MPPTh1}, \ref{MPPTh2} and \ref{MPPL}, and using multilinear interpolation (see \cite{GLLZ}), we have the following theorem:
\begin{thm}
Let $ \vec{\alpha} = (\alpha_1,\alpha_2,\ldots,\alpha_m) \in \{0,1\}^m$ and $ 1 < p_1, p_2, \ldots, p_m < \infty$ with $\displaystyle \sum_{j=1}^m \frac{1}{p_j} = \frac{1}{r}.$ Then
\begin{enumerate}[label = $(\alph*)$]
\item For $\vec{\alpha}  \neq (1,1,\ldots,1),$  $\displaystyle \left\Vert P^{\vec{\alpha}}(f_1,f_2,\ldots,f_m)\right\Vert_r \lesssim \prod_{j=1}^m\norm{f_j}_{p_j}.$
\item For $\sigma(\vec{\alpha}) \leq 1,$ $\displaystyle \left\Vert\pi_b^{\vec{\alpha}}(f_1,f_2,\ldots,f_m)\right\Vert_r \lesssim \norm{b}_{BMO^d}\prod_{j=1}^m\norm{f_j}_{p_j},$ if and only if $b \in BMO^d.$\\

\item For $\sigma(\vec{\alpha}) > 1,$ $\displaystyle\left\Vert\pi_b^{\vec{\alpha}}(f_1,f_2,\ldots,f_m)\right\Vert_r \leq C_b \prod_{j=1}^m\norm{f_j}_{p_j},$ if and only if $\displaystyle\sup_{I\in \D} \frac{\abs{\La b,h_I\Ra}}{\sqrt{\abs{I}}} < \infty.$
\end{enumerate}
In each of the above cases, the paraproducts are weakly bounded if $1\leq p_1, p_2, \ldots, p_m < \infty$.\\
\end{thm}

\section {Multilinear Haar multipliers and multilinear commutators}
\subsection{Multilinear Haar Multipliers}

\noindent
In this subsection we introduce multilinear Haar multipliers, and study their boundedness properties.\\

\noindent
\begin{dfn}
Given $\vec{\alpha} = (\alpha_1,\alpha_2, \ldots,\alpha_m) \in \{0,1\}^m,$ and a symbol sequence $\epsilon = \{\epsilon_I\}_{I\in\D},$ we define \textit{m-linear Haar multipliers} by
$$ T_\epsilon^{\vec{\alpha}} (f_1,f_2, \ldots,f_m) \equiv \sum_{I\in \D} \epsilon_I \prod_{j=1}^m f_j(I,\alpha_j) h_I^{\sigma(\vec{\alpha})}.$$
\end{dfn}
\noindent
\begin{thm}\label{MHMTh}
Let $\epsilon = \{\epsilon_I\}_{I\in\D}$ be a given sequence and let $\vec{\alpha} = (\alpha_1,\alpha_2, \ldots,\alpha_m) \in U_m.$  Let $1<p_1,p_2, \ldots,p_m<\infty$ with
$$\displaystyle \sum_{j=1}^m \frac{1}{p_j} = \frac{1}{r}.$$
Then $T_\epsilon^{\vec{\alpha}}$ is bounded from $L^{p_1}\times L^{p_2} \times \cdots\times L^{p_m}$ to $L^r$ if and only if $\norm{\epsilon}_\infty:= \displaystyle \sup_{I \in \D}\abs{\epsilon_I} < \infty.$\\
Moreover, $T_\epsilon^{\vec{\alpha}}$ has the corresponding weak-type boundedness if $1 \leq p_1,p_2, \ldots,p_m<\infty.$
\end{thm}

\noindent
\begin{proof}
To prove this lemma we use the fact that the linear Haar multiplier
$$T_\epsilon(f) = \sum_{I\in\D} \epsilon_I \La f,h_I\Ra h_I$$
is bounded on $L^p$ for all $1<p<\infty$ if $\norm{\epsilon}_\infty:= \displaystyle \sup_{I \in \D}\abs{\epsilon_I} < \infty,$ and that $\La T_\epsilon(f),h_I \Ra = \epsilon_I \La f,h_I\Ra.$\\

\noindent
By assumption $\sigma(\vec{\alpha})\geq 1$. Without loss of generality we may assume that $\alpha_i = 0$ if $1\leq i \leq \sigma(\vec{\alpha})$ and $\alpha_i = 1$ if $\sigma(\vec{\alpha}) < i \leq m.$ In particular, we have $\alpha_1 = 0.$  Then
 $$\epsilon_I f_1(I,\alpha_1) = \epsilon_I \La f_1,h_I\Ra = \La T_\epsilon(f_1),h_I \Ra = T_\epsilon(f_1)(I,\alpha_1).$$
\noindent
First assume that $\norm{\epsilon}_\infty:= \displaystyle \sup_{I \in \D}\abs{\epsilon_I} < \infty.$\\

\noindent
Then,
\begin{eqnarray*}
 \norm{T_\epsilon^{\vec{\alpha}} (f_1,f_2, \ldots,f_m)}_r 
&=& \left\Vert \sum_{I\in \D} \epsilon_I \prod_{j=1}^m f_j(I,\alpha_j) h_I^{\sigma(\vec{\alpha})}\right\Vert_r\\
&=& \left\Vert\sum_{I\in \D} T_\epsilon(f_1)(I,\alpha_1)\prod_{j=2}^m f_j(I,\alpha_j) h_I^{\sigma(\vec{\alpha})}\right\Vert_r\\
&=& \norm{P^{\vec{\alpha}}(T_\epsilon(f_1),f_2, \ldots,f_{m})}_r\\
& \lesssim & \norm{T_\epsilon(f_1)}_{p_1} \prod_{j=2}^m \norm{f_j}_{p_j}\\
& \lesssim & \prod_{j=1}^m\norm{f_j}_{p_j}.\\
\end{eqnarray*}

\noindent
Conversely, assume that $T_\epsilon^{\vec{\alpha}}$ : $L^{p_1}\times L^{p_2} \times \cdots\times L^{p_m}\rightarrow L^r$ is bounded, and  let $\sigma(\vec{\alpha}) = k.$ Recall that $\alpha_i = 0$ if $1\leq i \leq \sigma(\vec{\alpha}) = k$ and $\alpha_i = 1$ if $k=\sigma(\vec{\alpha}) < i \leq m.$ Taking $f_i = h_I$ if $1 \leq i \leq k$ and $f_i = \mathsf{1}_I$ if $k < i \leq m,$ we observe that\\

\begin{eqnarray*}
{\norm{T_\epsilon^{\vec{\alpha}}(f_1,f_2, \ldots,f_m)}}_r &=& \left(\int_{\R} \abs{\epsilon_I h_I^k(x)}^r dx\right)^{1/r} \\
&=&  \left(\frac{\abs{\epsilon_I}^r}{\abs{I}^{kr/2}}\int_{\R} \mathsf{1}_I(x) dx\right)^{1/r} \\
& = & \frac{\abs{\epsilon_I}}{\abs{I}^{k/2}} \abs{I}^{1/r} 
\end{eqnarray*}
and,
\begin{eqnarray*}
\prod_{j=1}^m\norm{f_j}_{p_j} &=& \prod_{i=1}^k \left(\int_{\R} \abs{ h_I(x)}^{p_i} dx \right)^{1/p_i}  \prod_{j=k+1}^m\left(\int_{\R} \abs{\mathsf{1}_I(x)}^{p_j} dx\right)^{1/p_j}\\
&=&  \prod_{i=1}^k\left(\frac{1}{\abs{I}^{p_i/2}} \int_{\R} \mathsf{1}_I(x) dx \right)^{1/p_i}\prod_{j=k+1}^m\left(\int_{\R} \mathsf{1}_I(x) dx \right)^{1/p_j}\\
& = & \prod_{i=1}^k\left(\frac{1}{\abs{I}^{1/2}} \abs{I}^{1/p_i}\right) \prod_{j=k+1}^m \abs{I}^{1/p_j}\\
& = & \frac{\abs{I}^{1/r}}{\abs{I}^{k/2}}
\end{eqnarray*}
Since $(f_1, f_2, \ldots,f_m) \in L^{p_1}\times L^{p_2} \times \cdots\times L^{p_m}$, the boundedness of $T_\epsilon$ implies that

$${\norm{T_\epsilon^{\vec{\alpha}}(f_1,f_2, \ldots,f_m)}}_r \leq  {\norm{T_\epsilon^{\vec{\alpha}}}}_{L^{p_1}\times \cdots\times L^{p_m} \rightarrow L^r} \prod_{j=1}^m\norm{f_j}_{p_j}.$$
That is,$$\frac{\abs{\epsilon_I}}{\abs{I}^{k/2}} \abs{I}^{1/r} \leq {\norm{T_\epsilon^{\vec{\alpha}}}}_{L^{p_1}\times \cdots\times L^{p_m}} \frac{\abs{I}^{1/r}}{\abs{I}^{k/2}},$$
for all $I \in \D.$ Consequently, $\norm{\epsilon}_\infty = \displaystyle\sup_{I\in \mathcal{D}} \abs{\epsilon_I} \leq {\norm{T_\epsilon^{\vec{\alpha}}}}_{L^{p_1}\times \cdots\times L^{p_m}} < \infty,$ as desired. \\
If $1 \leq p_1,p_2, \ldots,p_m<\infty,$ the weak-type boundedness of $T_\epsilon^{\vec{\alpha}}$ follows from Lemma \ref{WBL}.
\end{proof}

\subsection{Multilinear commutators}
\noindent
In this subsection we study boundedness properties of the commutators of $T_\epsilon^{\vec{\alpha}}$ with the multiplication operator $M_b$ when $b\in BMO^d.$ For convenience we denote the operator $M_b$ by $b$ itself. We are interested in the following commutators:
$$[b,T_\epsilon^{\vec{\alpha}}]_i(f_1,f_2,\ldots,f_m)(x) \equiv (T_\epsilon^{\vec{\alpha}}(f_1, \ldots, bf_i,\ldots,f_m) - bT_\epsilon^{\vec{\alpha}}(f_1,f_2,\ldots,f_m))(x)$$

\noindent
where $1\leq i \leq m.$\\

\noindent
Note that if $b$ is a constant function, $[b,T_\epsilon^{\vec{\alpha}}]_i(f_1,f_2,\ldots,f_m)(x) = 0$ for all $x.$ Our approach to study the boundedness properties of $[b,T_\epsilon^{\vec{\alpha}}]_i: L^{p_1}\times L^{p_2} \times \cdots\times L^{p_m}\rightarrow L^r$ with $1<p_1,p_2, \ldots,p_m < \infty$ and $\displaystyle \sum_{j=1}^m \frac{1}{p_j} = \frac{1}{r}$  for non-constant $b$ requires us to assume that $b \in L^p$ for some $p\in (1,\infty),$ and that $r > 1.$ However, this restricted unweighted theory turns out to be sufficient to obtain a weighted theory, which in turn implies the unrestricted unweighted theory of these multilinear commutators. We will present the weighted theory of these commutators in a subsequent paper.\\

\begin{thm}\label{boc}
Let $\vec{\alpha} = (\alpha_1,\alpha_2,\ldots,\alpha_m) \in U_m.$  If $b \in BMO^d \cap L^p$ for some $1<p<\infty$ and $\norm{\epsilon}_\infty := \sup_{I\in \mathcal{D}}  \abs{\epsilon_I}  < \infty,$  then each commutator $[b,T_\epsilon^{\vec{\alpha}}]_i$ is bounded from $L^{p_1}\times L^{p_2} \times \cdots\times L^{p_m}\rightarrow L^r$  for all $1<p_1,p_2, \ldots,p_m, r < \infty$ with 
$$\sum_{j=1}^m \frac{1}{p_j} = \frac{1}{r},$$
 with estimates of the form:
$$ \norm{[b,T_\epsilon^{\vec{\alpha}}]_i(f_1,f_2,\ldots,f_m)}_r \lesssim \norm{b}_{BMO^d}\prod_{j=1}^m\norm{f_j}_{p_j}.$$ 

\end{thm}

\noindent
\begin{proof} It suffices to prove boundedness of $[b,T_\epsilon^{\vec{\alpha}}]_1,$ as the others are identical. Moreover, we may assume that each $f_i$ is bounded and has compact support, since such functions are dense in the $L^p$ spaces.\\

\noindent
Writing $bf_1 = \pi_b(f_1) + \pi_b^*(f_1) + \pi_{f_1}(b)$ and using multilinearity of $T_\epsilon^{\vec{\alpha}}$, we have\\
$$T_\epsilon^{\vec{\alpha}} (bf_1, f_2, \ldots,f_m) 
= T_\epsilon^{\vec{\alpha}} (\pi_b(f_1),f_2, \ldots,f_m) + T_\epsilon^{\vec{\alpha}} (\pi_b^*(f_1),f_2, \ldots,f_m) + T_\epsilon^{\vec{\alpha}} (\pi_{f_1}(b),f_2, \ldots,f_m).$$

\noindent
On the other hand,
\begin{eqnarray*}
b T_\epsilon^{\vec{\alpha}} (f_1,f_2, \ldots,f_m) 
&=& \sum_{I \in \D}\epsilon_I \prod_{j=1}^m f_j(I,\alpha_j) h_I^{\sigma(\vec{\alpha})}\left(\sum_{J \in \D} \widehat{b}(J) h_J\right)\\
&=& \sum_{I \in \D}\epsilon_I \widehat{b}(I)\prod_{j=1}^m f_j(I,\alpha_j) h_I^{1+\sigma(\vec{\alpha})}\\
&& +\sum_{I \in \D}\epsilon_I \prod_{j=1}^m f_j(I,\alpha_j) h_I^{\sigma(\vec{\alpha})}\left(\sum_{I\subsetneq J} \widehat{b}(J) h_J\right)\\
&& +\sum_{I \in \D}\epsilon_I \prod_{j=1}^m f_j(I,\alpha_j) h_I^{\sigma(\vec{\alpha})}\left(\sum_{J\subsetneq I} \widehat{b}(J) h_J\right)\\
&=& \pi_b^{\vec{\alpha}} (f_1, \ldots, T_\epsilon(f_i), \ldots,f_m) \\
&& + \sum_{I \in \D}\epsilon_I \La b\Ra_I\prod_{j=1}^m f_j(I,\alpha_j) h_I^{\sigma(\vec{\alpha})} \\
&&+ \sum_{J\in\D} \widehat{b}(J) h_J\left(\sum_{J\subsetneq I}\epsilon_I \prod_{j=1}^m f_j(I,\alpha_j) h_I^{\sigma(\vec{\alpha})}\right)
\end{eqnarray*}
for some $i$ with $\alpha_i = 0.$ Indeed, some $\alpha_i$ equals 0 by assumption, and for such $i$, we have
$$T_\epsilon(f_i)(I,\alpha_i) = \widehat{T_\epsilon(f_i)}(I) = \epsilon_I \widehat{f_i}(I) = \epsilon_I f_i(I,\alpha_i) .$$

\noindent
For $(f_1,f_2,\ldots,f_m) \in L^{p_1}\times L^{p_2} \times \cdots\times L^{p_m},$ we have
\begin{eqnarray*}
\norm{T_\epsilon^{\vec{\alpha}} (\pi_b(f_1),f_2, \ldots,f_m)}_r &\lesssim & \norm{\pi_b(f_1)}_{p_1}\prod_{j=2}^m \norm{f_j}_{p_j}\\
&\lesssim & \norm{b}_{BMO^d}\prod_{j=1}^m\norm{f_j}_{p_j}
\end{eqnarray*}
\begin{eqnarray*}
\norm{T_\epsilon^{\vec{\alpha}} (\pi_b^*(f_1),f_2, \ldots,f_m)}_r &\lesssim & \norm{\pi_b^*(f_1)}_{p_1}\prod_{j=2}^m \norm{f_j}_{p_j}\\
&\lesssim & \norm{b}_{BMO^d}\prod_{j=1}^m\norm{f_j}_{p_j}.
\end{eqnarray*}
and,
\begin{eqnarray*}
\norm{\pi_b^{\vec{\alpha}} (f_1, \ldots,T_\epsilon(f_i), \ldots,f_m)}_r &\lesssim & \norm{b}_{BMO^d}\norm{f_1}_{p_1} \cdots \norm{T_\epsilon(f_i)}_{p_i}\cdots\norm{f_m}_{p_m}\\
&\lesssim& \norm{b}_{BMO^d}\prod_{j=1}^m\norm{f_j}_{p_j}.\\
\end{eqnarray*}

\noindent
So, to prove boundedness of $[b,T_\epsilon^{\vec{\alpha}}]_1$, is suffices to show similar control over the terms:
\begin{equation} \label{term1}\left\Vert \sum_{J\in\D} \widehat{b}(J) h_J\left(\sum_{J\subsetneq I}\epsilon_I \prod_{j=1}^m f_j(I,\alpha_j) h_I^{\sigma(\vec{\alpha})}\right)\right\Vert_r
\end{equation}
and,
\begin{equation}\label{term2}\left\Vert T_\epsilon^{\vec{\alpha}} (\pi_{f_1}(b),f_2, \ldots,f_m)- \sum_{I \in \D}\epsilon_I \La b\Ra_I \prod_{j=1}^m f_j(I,\alpha_j)  h_I^{\sigma(\vec{\alpha})}\right\Vert_r. 
\end{equation}

\noindent
\textbf{Estimation of} $(\ref{term1})$:\\ 

\noindent
Case I: $\sigma(\vec{\alpha})$ odd.\\
In this case,
$$T_\epsilon^{\vec{\alpha}} (f_1,f_2, \ldots,f_m) = \sum_{I \in \D}\epsilon_I \prod_{j=1}^m f_j(I,\alpha_j) h_I^{\sigma(\vec{\alpha})} = \sum_{I \in \D}\epsilon_I\abs{I}^{\frac{1-\sigma(\vec{\alpha})}{2}} \prod_{j=1}^m f_j(I,\alpha_j) h_I. $$
So, $$  \La T_\epsilon^{\vec{\alpha}} (f_1,f_2, \ldots,f_m), h_I \Ra h_I = \epsilon_I\abs{I}^{\frac{1-\sigma(\vec{\alpha})}{2}} \prod_{j=1}^m f_j(I,\alpha_j) h_I 
=  \epsilon_I \prod_{j=1}^m f_j(I,\alpha_j)h_I^{\sigma(\vec{\alpha})}.
$$
This implies that
\begin{eqnarray*}
(\ref{term1}) &=&\left\Vert \sum_{J\in\D} \widehat{b}(J) h_J\left(\sum_{J\subsetneq I}\La T_\epsilon^{\vec{\alpha}} (f_1,f_2, \ldots,f_m), h_I \Ra h_I\right)\right\Vert_r\\
&= & \left\Vert \sum_{J\in\D} \widehat{b}(J)\La T_\epsilon^{\vec{\alpha}} (f_1,f_2, \ldots,f_m) \Ra_J h_J\right\Vert_r\\
&=& \left\Vert \pi_b \left( T_\epsilon^{\vec{\alpha}} (f_1,f_2, \ldots,f_m) \right)\right\Vert_r\\
&\lesssim & \norm{b}_{BMO^d}\left\Vert T_\epsilon^{\vec{\alpha}} (f_1,f_2, \ldots,f_m) \right\Vert_r\\
&\lesssim & \norm{b}_{BMO^d}\prod_{j=1}^m\norm{f_j}_{p_j}.
\end{eqnarray*}

\noindent
Case II: $\sigma(\vec{\alpha})$ even.\\

\noindent
In this case at least two $\alpha_i's$ are equal to 0. Without loss of generality we may assume that $\alpha_1=0.$ Then denoting $T_\epsilon (f_1)$ by $g_1,$ $P^{(\alpha_2,\ldots,\alpha_m)}(f_2, \ldots,f_m)$ by $ g_2,$ and using the fact that
$$\La g_1 \Ra_J \La g_2 \Ra_J \mathsf{1}_J = \left(\sum_{J\subsetneq I}\widehat{g_1}(I) \La{g_2}\Ra_I h_I  +\sum_{J\subsetneq I}\La g_1\Ra_I \widehat{g_2}(I) h_I +\sum_{J\subsetneq I}\widehat{g_1}(I) \widehat{g_2}(I) h_I^2 \right)\mathsf{1}_J, $$ we have
\begin{eqnarray*}
&& \left\Vert \sum_{J\in\D} \widehat{b}(J) h_J\left(\sum_{J\subsetneq I}\epsilon_I \prod_{j=1}^m f_j(I,\alpha_j) h_I^{\sigma(\vec{\alpha})}\right)\right\Vert_r\\  
&=&\left\Vert \sum_{J\in\D} \widehat{b}(J) h_J\left(\sum_{J\subsetneq I}\widehat{g_1}(I) \widehat{g_2}(I) h_I^2 \right)\right\Vert_r\\
&= & \left\Vert \sum_{J\in\D} \widehat{b}(J) h_J\left(\La g_1 \Ra_J \La g_2 \Ra_J \mathsf{1}_J - \sum_{J\subsetneq I}\widehat{g_1}(I) \La{g_2}\Ra_I h_I  -\sum_{J\subsetneq I}\La g_1\Ra_I \widehat{g_2}(I) h_I \right)\right\Vert_r\\
&\leq & \left\Vert \sum_{J\in\D} \widehat{b}(J) \La g_1 \Ra_J \La g_2 \Ra_J h_J\right\Vert_r +\left\Vert \sum_{J\in\D} \widehat{b}(J) \La P^{(0,1)}(g_1,g_2) \Ra_J h_J \right\Vert_r\\
&& \qquad \qquad +\left\Vert \sum_{J\in\D} \widehat{b}(J) \La P^{(1,0)}(g_1,g_2) \Ra_J h_J\right\Vert_r \\
&\lesssim & \norm{b}_{BMO^d}\norm{g_1}_{p_1}\norm{g_2}_{q} + \norm{b}_{BMO^d}\norm{P^{(0,1)}(g_1,g_2)}_{r} +\norm{b}_{BMO^d}\norm{P^{(1,0)}(g_1,g_2)}_r\\
&\lesssim & \norm{b}_{BMO^d}\norm{g_1}_{p_1}\norm{g_2}_{q}\\
&\lesssim & \norm{b}_{BMO^d}\prod_{j=1}^m\norm{f_j}_{p_j}.
\end{eqnarray*}
where, $q$ is given by $\displaystyle \frac{1}{q} = \sum_{j=2}^m \frac{1}{p_j}.$ Here the last three inequalities follow from Theorems $\ref{MPPTh1}$ and $\ref{MPPTh2},$ and the fact that $\norm{g_1}_{p_1} = \norm{T_\epsilon (f_1)}_{p_1} \lesssim \norm{f_1}_{p_1}.$ \\

\noindent
\textbf{Estimation of} $(\ref{term2}):$\\

\noindent
Case I: $\alpha_1 = 0.$\\

\noindent
This case is easy as we observe that\\
\begin{eqnarray*}
&& T_\epsilon^{\vec{\alpha}} (\pi_{f_1}(b),f_2, \ldots,f_m)- \sum_{I \in \D}\epsilon_I \La b\Ra_I \prod_{j=1}^m f_j(I,\alpha_j) h_I^{\sigma(\vec{\alpha})}\\
&=& \sum_{I \in \D}\epsilon_I \widehat{\pi_{f_1}(b)}(I) \prod_{j=2}^m f_j(I,\alpha_j) h_I^{\sigma(\vec{\alpha})} - \sum_{I \in \D}\epsilon_I \La b\Ra_I \widehat{f_1}(I) \prod_{j=2}^m f_j(I,\alpha_j) h_I^{\sigma(\vec{\alpha})}\\
& =& \sum_{I \in \D}\epsilon_I \La b\Ra_I \widehat{f_1}(I)\prod_{j=2}^m f_j(I,\alpha_j) h_I^{\sigma(\vec{\alpha})} - \sum_{I \in \D}\epsilon_I \La b\Ra_I \widehat{f_1}(I) \prod_{j=2}^m f_j(I,\alpha_j) h_I^{\sigma(\vec{\alpha})}\\
&=& 0.
\end{eqnarray*}
So there is nothing to estimate.
\newpage
\noindent
Case II: $\alpha_1 = 1.$\\

\noindent
In this case,
\begin{eqnarray*}
&& T_\epsilon^{\vec{\alpha}} (\pi_{f_1}(b),f_2, \ldots,f_m) - \sum_{I \in \D}\epsilon_I \La b\Ra_I \prod_{j=1}^m f_j(I,\alpha_j) h_I^{\sigma(\vec{\alpha})}\\
&=& \sum_{I \in \D}\epsilon_I \La \pi_{f_1}(b) \Ra_I \prod_{j=2}^m f_j(I,\alpha_j) h_I^{\sigma(\vec{\alpha})} - \sum_{I \in \D}\epsilon_I \La b\Ra_I \La f_1\Ra_I \prod_{j=2}^m f_j(I,\alpha_j) h_I^{\sigma(\vec{\alpha})}\\
& =& \sum_{I \in \D}\epsilon_I \left(\La \pi_{f_1}(b) \Ra_I - \La b\Ra_I \La f_1\Ra_I \right)\prod_{j=2}^m f_j(I,\alpha_j) h_I^{\sigma(\vec{\alpha})}\\
\end{eqnarray*}

\noindent
Now,
\begin{eqnarray*}\La b \Ra_I \La f_1 \Ra_I \mathsf{1}_I &=& \sum_{I\subsetneq J}\widehat{b}(J) \La{f_1}\Ra_J h_J \mathsf{1}_I +\sum_{I\subsetneq J}\La b\Ra_J \widehat{f_1}(J) h_J \mathsf{1}_I+\sum_{I\subsetneq J}\widehat{b}(J) \widehat{f_1}(J) h_J^2\mathsf{1}_I\\
&=& \La \pi_b(f_1)\Ra_I \mathsf{1}_I +\La \pi_{f_1}(b)\Ra_I  \mathsf{1}_I+ \sum_{I\subsetneq J}\widehat{b}(J) \widehat{f_1}(J) h_J^2\mathsf{1}_I.
\end{eqnarray*}
Hence, $\La b \Ra_I \La f_1 \Ra_I \mathsf{1}_I -\La \pi_{f_1}(b)\Ra_I \mathsf{1}_I= \La \pi_b(f_1)\Ra_I \mathsf{1}_I + \displaystyle \sum_{I\subsetneq J}\widehat{b}(J) \widehat{f_1}(J) h_J^2 \mathsf{1}_I.$\\

\noindent
So we have 
\begin{eqnarray*}
&& T_\epsilon^{\vec{\alpha}} (\pi_{f_1}(b),f_2, \ldots,f_m)- \sum_{I \in \D}\epsilon_I \La b\Ra_I \prod_{j=1}^m f_j(I,\alpha_j) h_I^{\sigma(\vec{\alpha})}\\
&& = -\sum_{I \in \D}\epsilon_I \left( \La \pi_b(f_1)\Ra_I \mathsf{1}_I + \sum_{I\subsetneq J}\widehat{b}(J) \widehat{f_1}(J) h_J^2\right)\prod_{j=2}^m f_j(I,\alpha_j) h_I^{\sigma(\vec{\alpha})}\\
&& = -\sum_{I \in \D}\epsilon_I  \La \pi_b(f_1)\Ra_I \prod_{j=2}^m f_j(I,\alpha_j) h_I^{\sigma(\vec{\alpha})} \\
&& \hspace{1in} -\sum_{I \in \D}\epsilon_I \left( \sum_{I\subsetneq J}\widehat{b}(J) \widehat{f_1}(J) h_J^2\right)\prod_{j=2}^m f_j(I,\alpha_j) h_I^{\sigma(\vec{\alpha})}\\
&& = - T_\epsilon(\pi_b(f_1),f_2,\ldots,f_m) - \sum_{J \in \D}\widehat{b}(J) \widehat{f_1}(J) h_J^2 \left(\sum_{I\subsetneq J}\epsilon_I \prod_{j=2}^m f_j(I,\alpha_j) h_I^{\sigma(\vec{\alpha})} \right).\\
\end{eqnarray*}

\noindent
Since $$\norm{T_\epsilon(\pi_b(f_1),f_2,\ldots,f_m)}_r \lesssim \norm{\pi_b(f_1)}_{p_1}\prod_{j=2}^m f_j(J,\alpha_j) \lesssim \norm{b}_{BMO^d}\prod_{j=1}^m\norm{f_j}_{p_j},$$
we are left with controlling $$\left\Vert \sum_{J \in \D}\widehat{b}(J) \widehat{f_1}(J) h_J^2 \left(\sum_{I\subsetneq J}\epsilon_I \prod_{j=2}^m f_j(I,\alpha_j) h_I^{\sigma(\vec{\alpha})} \right)\right\Vert_r. $$

\noindent
For this we observe that
$$\left\Vert T_\epsilon^{(\alpha_2, \ldots,\alpha_m)}(f_2, \ldots,f_m) \right\Vert_q \lesssim \prod_{j=2}^m \norm{f_j}_{p_j},$$
and that
 \begin{eqnarray*}
\pi_b^*(f_1)\; T_\epsilon^{(\alpha_2, \ldots,\alpha_m)}(f_2, \ldots,f_m) 
&=& \sum_{J \in \D}\widehat{b}(J) \widehat{f_1}(J) h_J^2 \left( \sum_{I\subsetneq J}\epsilon_I \prod_{j=2}^m f_j(I,\alpha_j) h_I^{\sigma(\vec{\alpha})}\right)\\
&& + \sum_{J \in \D}\epsilon_J \widehat{b}(J) \widehat{f_1}(J)  \prod_{j=2}^m f_j(J,\alpha_j) h_J^{2+ \sigma(\vec{\alpha})}
\\
&& +\sum_{J \in \D}\widehat{b}(J) \widehat{f_1}(J) h_J^2\left(\sum_{J\subsetneq I}\epsilon_I \prod_{j=2}^m f_j(I,\alpha_j) h_I^{\sigma(\vec{\alpha})} \right)
\end{eqnarray*}

\noindent
Now, following the same technique we used to control $(\ref{term1}),$ we obtain
 $$\displaystyle \left\Vert\sum_{J \in \D}\widehat{b}(J) \widehat{f_1}(J) h_J^2\left(\sum_{J\subsetneq I}\epsilon_I \prod_{j=2}^m f_j(I,\alpha_j) h_I^{\sigma(\vec{\alpha})}\right)\right\Vert_r\lesssim \norm{b}_{BMO^d}\prod_{j=1}^m\norm{f_j}_{p_j}.$$
We also have
\begin{eqnarray*}
\left\Vert \pi_b^*(f_1)\; T_\epsilon^{(\alpha_2, \ldots,\alpha_m)}(f_2, \ldots,f_m) \right\Vert_r 
&\leq& \left\Vert\pi_b^*(f_1)\right\Vert_{p_1} \left\Vert T_\epsilon^{(\alpha_2, \ldots,\alpha_m)}(f_2, \ldots,f_m)\right\Vert_q \\  
&\lesssim& \norm{b}_{BMO^d}\prod_{j=1}^m\norm{f_j}_{p_j}
\end{eqnarray*}
and, $$ \left\Vert\sum_{J \in \D}\epsilon_J \widehat{b}(J) \widehat{f_1}(J)  \prod_{j=2}^m f_j(J,\alpha_j) h_J^{2+ \sigma(\vec{\alpha})}\right\Vert_r \lesssim \norm{b}_{BMO^d}\prod_{j=1}^m\norm{f_j}_{p_j}.$$.

\noindent
So we conclude that $$ \left\Vert\sum_{J \in \D}\widehat{b}(J) \widehat{f_1}(J) h_J^2 \left(\sum_{I\subsetneq J}\epsilon_I \prod_{j=2}^m f_j(I,\alpha_j) h_I^{\sigma(\vec{\alpha})} \right)\right\Vert_r\lesssim\norm{b}_{BMO^d}\prod_{j=1}^m\norm{f_j}_{p_j}.$$
Thus we have strong type boundedness of $$[b,T_\epsilon^{\vec{\alpha}}]_1 \rightarrow L^{p_1}\times L^{p_2} \times \cdots\times L^{p_m}\rightarrow L^r$$  for all $1<p_1,p_2, \ldots,p_m, r < \infty$ with 
$$\displaystyle \sum_{j=1}^m \frac{1}{p_j} = \frac{1}{r}.$$ 
\end{proof} 

\noindent
In the next theorem, we show that BMO condition is necessary for the boundedness of the commutators.\\

\noindent
\begin{thm}\label{bmonecessity}
Let $\vec{\alpha} = (\alpha_1,\alpha_2,\ldots,\alpha_m) \in U_m,$ and  $1<p_1,p_2, \ldots,p_m, r < \infty$ with 
$$\sum_{j=1}^m \frac{1}{p_j} = \frac{1}{r}.$$ Assume that for given $b$ and $i$, 
\begin{equation}
\norm{[b,T_\epsilon^{\vec{\alpha}}]_i(f_1,f_2,\ldots,f_m)}_r \leq C_\epsilon \prod_{j=1}^m\norm{f_j}_{p_j}, 
\label{eq:bd}
\end{equation}
for every bounded sequence $\epsilon = \{\epsilon_I\}_{I\in \D},$ and for all $f_i \in L^{p_i}.$ Then $b\in BMO^d.$
\end{thm}

\noindent
\begin{proof}
Without loss of generality we may assume that $i=1.$ Fix $I_0 \in \D$ and let $\epsilon = \{\epsilon_I\}_{I\in \D}$ with $\epsilon_I =1$ for all $I\in \D.$  \\
\noindent
\textbf{Case I:} $\alpha_1 = 0, \sigma(\vec{\alpha}) = 1.$\\
\noindent
Take $f_1 = \mathsf{1}_{I_0}$ and $f_i = h_{I_0^{(1)}}$ for  $i>1$, where $I_0^{(1)}$ is the parent of $I_0.$ Then,
$$T_\epsilon^{\vec{\alpha}}(f_1,f_2,\ldots,f_m)) = \sum_{I\in \D} \La \mathsf{1}_{I_0}, h_I\Ra \La h_{I_0^{(1)}}\Ra_I^{m-1} h_I=0,$$
and, \begin{eqnarray*} T_\epsilon^{\vec{\alpha}}(bf_1, f_2, \ldots, \ldots,f_m) &=& \sum_{I\in \D} \La b\mathsf{1}_{I_0}, h_I\Ra \La h_{I_0^{(1)}}\Ra_I^{m-1} h_I\\
&=& \sum_{I\subseteq I_0} \La b\mathsf{1}_{I_0}, h_I\Ra \left( \frac{K(I_0,I_0^{(1)})}{\sqrt{\left\vert{I_0^{(1)}}\right\vert}}\right)^{m-1}h_I\\
&=& \left( \frac{K(I_0,I_0^{(1)})}{\sqrt{\left\vert{I_0^{(1)}}\right\vert}}\right)^{m-1}\sum_{I\subseteq I_0} \La b, h_I\Ra h_I,
\end{eqnarray*}
where $ K(I_0,I_0^{(1)})$ is either $1$ or $-1$ depending on whether $I_0$ is the right or left half of $I_0^{(1)}.$ \\
\noindent
For the second to last equality we observe that, if $I$ is not a proper subset of $I_0^{(1)},$ $ \La h_{I_0^{(1)}}\Ra_I = 0,$ and that if  $I$ is a proper subset of $I_0^{(1)}$ but is not a subset of $I_0$, then  $\La b\mathsf{1}_{I_0}, h_I\Ra =0.$ Moreover, for $I \subseteq I_0,$ $\La b\mathsf{1}_{I_0}, h_I\Ra = \int_\R{b\mathsf{1}_{I_0} h_I} = \int_\R{b h_I} = \La b, h_I\Ra.$\\

\noindent
Now from inequality \eqref{eq:bd}, we get
$$ \left\Vert \left( \frac{K(I_0,I_0^{(1)})}{\sqrt{\left\vert{I_0^{(1)}}\right\vert}}\right)^{m-1}\sum_{I\subseteq I_0} \La b, h_I\Ra h_I \right\Vert_r  \leq C_\epsilon \abs{I_0}^{\frac{1}{p_1}} \prod_{i=2}^{m}\frac{\abs{I_0^{(1)}}^{\frac{1}{p_i}}} {\sqrt{\abs{I_0^{(1)}}}}$$
$$i.e. \quad  \left\Vert \sum_{I\subseteq I_0} \La b, h_I\Ra h_I \right\Vert_r \leq 2^{\frac{1}{p_2}+ \cdots+\frac{1}{p_m}} C_\epsilon \abs{I_0}^{\frac{1}{r}}.$$
Thus for every $I_0 \in \D,$ $$\frac{1}{\abs{I_0}^{\frac{1}{r}}}\left\Vert \sum_{I\subseteq I_0} \La b, h_I\Ra h_I \right\Vert_r \leq 2^{\frac{1}{p_2}+ \cdots+\frac{1}{p_m}} C_\epsilon ,$$ and hence $b \in BMO^d.$\\

\noindent
\textbf{Case II:} $\alpha_1 \neq 0$ \, or \, $\sigma(\vec{\alpha}) > 1.$\\
\noindent
Taking    
$f_i = 
     \begin{cases}
       h_{I_0}, &\text{if }\alpha_i = 0\\
       \mathsf{1}_{I_0}, \;\;\; & \text{if }\alpha_i = 1,\\
            \end{cases}
$\; we observe that
$$T_\epsilon^{\vec{\alpha}}(f_1,f_2,\ldots,f_m)) = h_{I_0}^{\sigma(\vec{\alpha})} \;\;\text{ and } \;\;\;T_\epsilon^{\vec{\alpha}}(bf_1, f_2, \ldots, \ldots,f_m) = (bf_1)(I_0, \alpha_1)h_{I_0}^{\sigma(\vec{\alpha})}. $$
\noindent
If $\alpha_1 = 0, $ 
 $$ (bf_1)(I_0, \alpha_1) = {bh_{I_0}}(I_0, 0) = \widehat{bh_{I_0}}(I_0)  = \int_\R{ bh_{I_0}h_{I_0}} = \frac{1}{\abs{I_0}}\int_\R{ b \mathsf{1}_{I_0}} = \La b \Ra_{I_0}.$$
\noindent
If $\alpha_1 = 1,$
 $$ (bf_1)(I_0, \alpha_1) = {b\mathsf{1}_{I_0}}(I_0, 1) = \La {b\mathsf{1}_{I_0}}\Ra_{I_0} = \La {b}\Ra_{I_0}.$$
\noindent
So in each case,
\begin{eqnarray*}
\norm{[b,T_\epsilon^{\vec{\alpha}}]_1(f_1,f_2,\ldots,f_m)}_r &=& \left\Vert{bT_\epsilon^{\vec{\alpha}}(f_1,f_2,\ldots,f_m)) - T_\epsilon^{\vec{\alpha}}(bf_1, f_2, \ldots, \ldots,f_m)}\right\Vert_r\\
&=& \left\Vert{b h_{I_0}^{\sigma(\vec{\alpha})} - \La b \Ra_{I_0} h_{I_0}^{\sigma(\vec{\alpha})}}\right\Vert_r\\
&=& \left\Vert{(b  - \La b \Ra_{I_0}) h_{I_0}^{\sigma(\vec{\alpha})}}\right\Vert_r\\
&=& \frac{1}{(\sqrt{\abs{I_0}})^{\sigma(\vec{\alpha})}}\norm{(b  - \La b \Ra_{I_0}) \mathsf{1}_{I_0}}_r.\\
\end{eqnarray*}
\noindent
On the other hand, $$\prod_{j=1}^m\norm{f_j}_{p_j} = \frac{1}{(\sqrt{\abs{I_0}})^{\sigma(\vec{\alpha})}} \abs{I_0}^{\frac{1}{p_1}+ \cdots + \frac{1}{p_m}} = \frac{1}{(\sqrt{\abs{I_0}})^{\sigma(\vec{\alpha})}} \abs{I_0}^{\frac{1}{r}}. $$
Inequality \eqref{eq:bd} then gives $$ \frac{1}{(\sqrt{\abs{I_0}})^{\sigma(\vec{\alpha})}}\norm{(b  - \La b \Ra_{I_0}) \mathsf{1}_{I_0}}_r \leq C_\epsilon \frac{1}{(\sqrt{\abs{I_0}})^{\sigma(\vec{\alpha})}} \abs{I_0}^{\frac{1}{r}}$$
$$ \text{i.e. }  \quad \frac{1}{\abs{I_0}^{\frac{1}{r}}}\norm{(b  - \La b \Ra_{I_0}) \mathsf{1}_{I_0}}_r \leq C_\epsilon. $$
Since this is true for any $I_0 \in \D$, we have $b\in BMO^d.$
\end{proof}

\noindent
Combining the results from Theorems \ref{boc} and \ref{bmonecessity}, we have the following characterization of the dyadic BMO functions. Note that if $\epsilon_I = 1$ for every $I \in \D$, we have $T_\epsilon^{\vec{\alpha}} = P^{\vec{\alpha}},$ and that in the proof of Theorem \ref{bmonecessity}, only the boundedness of $[b, T_\epsilon^{\vec{\alpha}}]_i$ for $\epsilon$ with $\epsilon_I = 1$ for all $I\in \D$ was used to show that $b\in BMO^d.$\\
\begin{thm}
Let $\vec{\alpha} = (\alpha_1,\alpha_2,\ldots,\alpha_m) \in U_m,$ $1\leq i \leq m,$ and $1<p_1,p_2, \ldots,p_m, r < \infty$ with 
$$\sum_{j=1}^m \frac{1}{p_j} = \frac{1}{r}.$$ Suppose $b \in L^p$ for some $p \in (1,\infty).$ Then the following two statements are equivalent.
\begin{enumerate}[label = $(\alph*)$]
\item  $b\in BMO^d.$\\
\item  $\displaystyle [b,T_\epsilon^{\vec{\alpha}}]_i:L^{p_1}\times L^{p_2} \times \cdots\times L^{p_m}\rightarrow L^r $ is bounded for every bounded sequence $\epsilon = \{\epsilon_I\}_{I\in \D}.$
\end{enumerate}

\noindent
In particular,  $b\in BMO^d$ if and only if
$[b,P^{\vec{\alpha}}]_i:L^{p_1}\times L^{p_2} \times \cdots\times L^{p_m}\rightarrow L^r$ is bounded.\\
\end{thm}




\begin{bibdiv}
\begin{biblist}

\normalsize
\bib{BMNT}{article}{
title={Bilinear paraproducts revisited},
  author={B{\'e}nyi, {\'A}.}, 
	author={Maldonado, D.},
	author={Nahmod, A. R.},
	author={Torres, R. H.},
  journal={Mathematische Nachrichten},
  volume={283},
  number={9},
  pages={1257--1276},
  year={2010},
  publisher={Wiley Online Library}}

\bib{Bla}{article}{
author={Blasco, O.},
title={Dyadic BMO, paraproducts and Haar multipliers},
journal={Contemp. Math., Vol 445, Amer. Math. Soc., Providence, RI,},
pages={11-18, MR 2381883}}

\bib{CRW}{article}{
title={Factorization theorems for Hardy spaces in several variables},
  author={Coifman, R.R.}, 
	author={Rochberg, R.},
	author={Weiss, G.},
	journal={Ann. of Math.},
  volume={103},
  pages={611-635},
  year={1976}}
	
\bib{GLLZ}{article}{
author={Grafakos,L.},
author={Liu, L.},
author={Lu, S},
author={Zhao,F.},
title={The multilinear Marcinkiewicz interpolation theorem revisited: The behavior of the constant},
journal={J. Funct. Anal.},
volume={262},
year={2012},
pages={2289-2313}}

\bib{GT}{article}{
author={Grafakos,L.},
author={Torres, R.H.},
title={Multilinear Calder$\acute{\text{o}}$n-Zygmund theory},
journal={Adv. Math.},
volume={165},
year={2002},
number={1},
pages={124-164.}}

\bib{HR}{article}{title={Interpolation by the real method between $BMO$, $L^\alpha (0 < \alpha < \infty)$ and $H^\alpha (0 < \alpha < \infty)$}, author = {Hanks, R.}, journal={Indiana Univ. Math. J.}, volume ={26}, number ={4}, pages={679-689}, year = {1977}}

\bib{Hyt}{article}{
author={Hyt$\ddot{\text{o}}$nen ,Tuomas P.},
title={Representation of singular integrals by dyadic operators, and the $A_2$ theorem},
journal={arXiv:1108.5119},
year={2011}}
	
\bib{Jan}{article}{
author={Janson, S.},
title={BMO and commutators of martingale transforms},
journal={Ann. Inst. Fourier},
volume={31},
number = {1},
year={1981},
pages={265-270}}

\bib{JN}{article}{title={On functions of bounded mean oscillation}, author = {John, F.},author = {Nirenberg, L.}, journal={Comm. Pure Appl. Math.}, volume = {14} year = {1961}, page ={415–426}}

 \bib{LOPTT}{article}{
title={New maximal functions and multiple weights for the multilinear Calder$\acute{\text{o}}$n-Zygmund theory},
  author={Lerner, A.K.}, 
	author={Ombrosi, S.},
	author={P$\acute{\text{e}}$rez, C.},
	author={Torres, R. H.},
	author={Trujillo-Gonz$\acute{\text{a}}$lez, R.},
  journal={Adv. in Math.},
  volume={220},
  number={4},
  pages={1222--1264},
  year={2009},
  publisher={Wiley Online Library}}
	
	\bib{Per}{article}{
author={Pereyra, M.C.},
title={Lecture notes on dyadic harmonic analysis},
journal={Contemporary Mathematics},
volume={289},
date={2001},
pages={1-60}}

\bib{SE}{book}{title={Harmonic Analysis: Real Variable Methods, Orthogonality, and Oscillatory Integrals}, author = {Stein, E. M.}, publisher ={Princeton Univ. Press, Princeton}, year = {1993}}

\bib{Treil}{article}{
author={Treil, S.},
title={Commutators, paraproducts and BMO in non-homogeneous martingale settings},
journal={http://arxiv.org/pdf/1007.1210v1.pdf}}

\end{biblist}
\end{bibdiv}
\end{document}